\documentclass[12pt,twoside%,draft%
]{article}

\usepackage{graphpap}
\usepackage{epsfig}

\usepackage{amscd}

\usepackage{amsmath}
\usepackage{amssymb}
\usepackage{amsthm}

\makeatletter
\def\@seccntformat#1{\csname the#1\endcsname{.}\hskip .5em}

\renewcommand{\section}{\@startsection {section}{1}{\z@}%
                                {-4.5ex \@plus -1ex \@minus-.2ex}%
                                   {2.3ex \@plus.2ex}%
                               {\reset@font\Large\scshape\centering}}

	\renewenvironment{proof}[1][\proofname]{\par					   
	  \normalfont												   
	  \topsep6\p@\@plus6\p@	\trivlist							   
	  \item[\hskip\labelsep\bfseries							   
		#1\@addpunct{.}]\ignorespaces							   
	}{%															   
	  \qed\endtrivlist											   
	}		

\makeatother

\normalbaselines

\newtheorem{thm}{Theorem}[section]
\newtheorem{lm}[thm]{Lemma}

\theoremstyle{definition}

\theoremstyle{remark}

\numberwithin{equation}{section}
\newcommand{\f}{\varphi}
\makeatletter

\def\multilimits@{\bgroup
  \Let@
  \restore@math@cr
  \default@tag
 \baselineskip\fontdimen10 \scriptfont\tw@
 \advance\baselineskip\fontdimen12 \scriptfont\tw@
 \lineskip\thr@@\fontdimen8 \scriptfont\thr@@
 \lineskiplimit\lineskip
 \vbox\bgroup\ialign\bgroup\hfil$\m@th\scriptstyle{##}$\hfil\crcr}  
\def\Sb{_\multilimits@}
\def\Sp{^\multilimits@}
\def\endSb{\crcr\egroup\egroup\egroup}

\makeatother

\newcommand{\supp}{\operatorname{supp}}

\newcommand{\al}{\alpha}

\newcommand{\cz}{Calder\'{o}n--Zygmund\ }

\newcommand{\e}{\varepsilon}

\newcommand{\dist}{\operatorname{dist}}

\newcommand{\pd}{\partial}
\newcommand{\Om}{\Omega}
\newcommand{\om}{\omega}
\newcommand{\bP}{\mathbb{P}}
\newcommand{\E}{\mathbb{E}}
\newcommand{\F}{\Phi}

\newcommand{\R}{\mathbb{R}^d}

\newcommand{\SSS}{\mathcal{S}}

\newcommand{\natwo}{[w]_{A_2}}
\newcommand{\naone}{[w]_{A_1}}

\newcommand{\nTtwo}{\|T\|_{L^2(w)\rightarrow L^2(w)}}
\newcommand{\nTtwoinf}{\|T\|_{L^2(w)\rightarrow L^{2,\infty}(w)}}

\newcommand{\nTtwoinfinv}{\|T'\|_{L^2(w^{-1})\rightarrow L^{2,\infty}(w^{-1})}}
\newcommand{\wI}{\langle w\rangle_I\langle w^{-1}\rangle_I}
\newcommand{\ww}{w^{-1}}
\newcommand{\eps}{\epsilon}
\newcommand{\ltwow}{L^2(wdx)}
\newcommand{\ltwowinv}{L^2(w^{-1}dx)}
\newcommand{\HHI}{\mathcal{H}_I^{\mu}}
\newcommand{\HHJ}{\mathcal{H}_J^{\nu}}
\newcommand{\HH}{\mathcal{H}}
\newcommand{\PP}{\mathcal{P}}

\newcommand{\PPI}{\PP_{I}}
{\end{list}}

\begin{document}

\title{On $A_2$ conjecture and corona decomposition of weights}
\author{ Carlos P\'erez\,, \thanks{Universidad de Sevilla}
Sergei Treil
\thanks{Dept. of Mathematics, Brown University}
 \,\,and Alexander Volberg
\thanks{Department of Mathematics, Michigan State University, East
Lansing, MI 48824, USA; 
}{ }\thanks{All  authors are  supported by the Spanish Research Council grant}
{ } 
{ }}
\date{}
\maketitle
%%%%
%\begin{center}  
%	{\bf First draft}
%\end{center}  
%%%%

%\section*{Notation}
%
%\addcontentsline{toc}%
%{section}{{\protect\numberline{}Notation}}

%\begin{abstract}

%\end{abstract}

\begin{center}
         Por acahar
         \end{center}

\tableofcontents

 \pagestyle{headings}
 \renewcommand{\sectionmark}[1]%
        {\markboth{}{\hfill\thesection{.}\ #1\hfill}}

         \renewcommand{\subsectionmark}[1]{\relax}
\setcounter{section}{0}
\section{Introduction and historical background of the problem}
\label{Intr}

We consider here a  problem of finding the sharp estimate for the boundedness of an arbitrary Calder\'on-Zygmund operator in $L^2(w)$, $w\in A_2$. In the 70«s Hunt--Muckenhoupt--Wheeden found a wonderfully simple characterization of weights for which the Hilbert transform is bounded from  $L^2(w)$ to itself. The problem had prediction theory background, because for a scalar stationary stochastic process weight has a meaning of its spectral measure density, and the boundedness of the Hilbert transform has a meaning of positive angle between the Past and a Future of the process--a good property of the processes, sort of their regularity. 

As such the problem of the boundedness of the Hilbert transform has been already attacked by Helson--Szeg\"o and Helson--Sarason. They also obtained a chracterization for the boundedness of the Hilbert transform in a weighted $L^2$. The answers were equivalent (of course) but totally different. Till nowdays nobody knows how to  obtain directly Helson--Szeg\"o condition from Hunt--Muckenhoupt--Wheeden condition. What we are doing below has some very vague flavor of going in this direction. Notice that the Helson--Szeg\"o--Sarason approach was developed for $p\neq 2$ by Cotlar--Sadosky in \cite{CS1}--\cite{CS5}.
In the 80«s a new point of view was introduced by Sawyer \cite{Saw1}, \cite{Saw2}, his treatment was concentrated on positive operators (the Hilbert transform is not one of them), and he introduced the test conditions: to check the boundedness of a certain class of (positive) operators it turned out to be sufficient to check the uniform boundedness on a (non-linear) family of test functions, usually a collection of characteristic functions of some sort. Simultaneously in the 80«s David and Journe \cite{DJ1}, \cite{DJ2} built a theory of \cz operators (here the Hilbert transform belongs) based on so-called $T1$ theorem. A closer look shows (but to the best of our knowledge nobody then made this closer look) that $T1$ theorem is exactly Sawyer`s test conditions. The difference was that there was no weight (life is easier), this was Lebesgue measure theory, but the operators were not positive, but rather singular (life is harder). Later $T1$ theory was proved to be fine not only  for Lebesgue measure, but still the measure should have some smoothness: this was done by Christ in \cite{Ch}. At the end of the 90«s a nonhomogeneous measures were included into $T1$ theory: see \cite{NTV1}--\cite{NTV7}, \cite{VolLip}, \cite{To1}, \cite{To2}.

At the same time at the beginning of   90«s several important papers appeared, which showed {\it how bounded} is the \cz (or maximal) operator if Hunt--Muckenhoupt--Wheeden conditions are satisfied. It was the return to the 70«s but on a new turn of the spiral. The questions of {\it sharp} weighted estimates appeared and seemed to be interesting not only for their own sake but mostly because they were needed by a) multiparameter Harmonic Analysis, b) sharp and especially critical exponent estimates for certain elliptic PDE. The first {\it sharp} estimate was obtained by Buckley \cite{Buck1}, he proved that the  $L^{2}(w)\rightarrow L^{2}(w)$ norm of the Hardy--Littlewood maximal operator grows at most as the first degree of the so-called $A_{2}$ norm $[w]_{A_{2}}$ of the weight. The proof was not easy. Now there exists a proof due to Lerner which takes only several lines. For \cz operators (namely for the Hilbert transform and such) R. Fefferman and J. Pipher \cite{FP} got a linear estimate in terms of $A_{1}$ norm of weight (it is another important characteristic, see below).

Then people started to consider not \cz operators, not positive operators of maximal type or potential type, but their models on dyadic lattice. The simplest and most well used singular dyadic operator is called Martingale transform (see, e. g., \cite{Bo1}, \cite{Bo2}, \cite{Bu1}): for it, the sharp linear estimate in terms of $A_{2}$ norm $[w]_{A_{2}}$ of the weight was obtained by Wittwer in \cite{Wit}. The interesting feature of her proof was that she used as a template a two-weight Martingale transform estimate of \cite{NTV-2w}.
Both \cite{NTV-2w} and  \cite{Wit} are {\it Bellman function proofs}. These two things: the use of two-weight approach (notoriously difficult for \cz operators), and the use of the Bellman function technique became the features of sharp weighted line of research. 

Explanation may be the following: if one wants a sharp estimate, one is in a paradoxical situation: one should  not use the good properties of weights, but one must use them! The exit is like that: use the good property but only once. The rest of the proof should be working for very bad measures (weights). This is how nonhomogeneous Harmonic Analysis and two-weights estimates come into play probably. We will see this below.

The Bellman proofs persisted, and in \cite{PetmV} the first for-real \cz operator got a sharp estimate by the first power of the norm $[w]_{A_{2}}$. This was the Ahlfors--Beurling operator, and its sharp weighted estimate allowed the authors to solve a problem of Iwaniec on a borderline regularity of Beltrami PDE.

The Hilbert transform turned out to be more difficult to treat, but in \cite{Petm1} Petermichl proved the linear estimate in terms of norm $[w]_{A_{2}}$ for the Hilbert transform as well. Then in \cite{Petm2} she did this for the Riesz transforms. The Ahlfors--Beurling operator is the averaging of Martingale transforms (see \cite{DV}), and this being established, the result of \cite{PetmV} became the corollary of \cite{Wit}. On the other hand, as Petermichl showed in \cite{Petm3}, the Hilbert transform is the averaging of the next-in-complexity dyadic operators: dyadic shifts. There are many more and more complex dyadic shifts, the linear estimate for all of them in terms of $[w]_{A_{2}}$ was shown recently in a very interesting paper of Lacey--Petermichl--Reguera \cite{LPR}. And then another proof appeared in  almost impossibly simple and beautiful papers of Cruz-Uribe, Martell, and the first author \cite{CUMP1}, \cite{CUMP2} and \cite{CPP}. They used an extremely beautiful ``formula`` by Lerner \cite{Ler1}.

So now we have a linear in terms of $[w]_{A_{2}}$ estimate for all dyadic shifts and all their averages, which is a subclass of quite smooth \cz operators. However, the general \cz operator is not a simple average of dyadic shifts. Moreover, dyadic shifts have ``depth`` $\tau$, which is the measure of their complexity. It is easy to get an estimate of their norms exponential in $\tau$. But this is bad if we want to give a linear in $[w]_{A_{2}}$ estimate for all \cz operators.

So we naturally come to the question to obtain a linear in $[w]_{A_{2}}$ estimate for all  \cz operators.  We almost get it. The reader can see this below.

\section{Main results}
\label{results}

In what follows $w$ is a weight in $A_2$, which as we know means
\begin{equation}
\label{a2}
\natwo:=\sup_I\wI <\infty\,,
\end{equation}
the quantity $\natwo$ will be called the ``norm" of the weight.  Operator $T$ will be always a {\bf bounded} operator in $L^2(\R)$ with Lebesgue measure such that
\begin{equation}
\label{CZ1}
(Tf,g)=\int K(x,y) f(y)g(x)dy dx
\end{equation}
for all nice $f,g$ having {\bf disjoint} supports. Here $K(x,y)$ denotes the kernel of the operator and it will be always Calder\'on--Zygmund (CZ) kernel. That means
$$
|K(x,y)|\le \frac{1}{|x-y|^d}\,,\,
$$
\begin{equation}
\label{CZ2}
|K(x,y)-K(x',y)| +|K(y,x)-K(y,x')|\le \frac{|x-x'|^{\epsilon}}{|x-y|^{d+\epsilon}}\,, |x-x'|\le \frac12 |x-y|\,.
\end{equation}

\noindent Notice that $K$ does not define $T$ uniquely, the identity operator and all operators of multiplication on a bounded function have the same kernel $K=0$. Anyhow, such operators are called Calder\'on--Zygmund operators (boundedness in $L^2$ with respect to Lebesgue measure and abovementioned properties of the kernel). So in our definition the identity is also a \cz operator, but of course a non-interesting one. By $T'$ we understand the corresponding transposed operator, its kernel is $K(y,x)$.

In what follows $C,c$ with indices denote absolute constants and constants depending on $d$ and $\epsilon$ only.

We are going to prove two main results.

\begin{thm}
\label{strongweak}
$\nTtwo \le c_1 \natwo +c_2( \nTtwoinf+\nTtwoinfinv)$\,.
\end{thm}

\noindent{\bf Remarks.} 1. Obviously $\nTtwoinf\le \nTtwo$. For all ``interesting" \cz operators also $\natwo\le c\nTtwo$. So the theorem gives a ``formula" for the norm. 

\noindent 2. It is a bit amazing what it ``almost" says: if \cz operator is of weak type $(2,2)$  then it is of the strong type $(2,2)$. Moreover, its weak type norm coincides (up to a constant) with a strong type norm! The last outrageous remark is basically true for all ``interesting" \cz operators, if we agree to call interesting those for which $\natwo\le \nTtwoinf$!

\bigskip

Our next result gives a pretty good (almost perfect) estimate of $\nTtwoinf$ (and thus of $\nTtwo$ by Theorem \ref{strongweak}). Here $T$ is {\bf any} operator. It is an abstract theorem. Along with $A_2$ class we need $A_1$. The weight is said to belong to $A_1$ if for any cube $I$

\begin{equation}
\label{a1}
\langle w\rangle_I \le B\,\inf_{x\in I} w(x)\,.
\end{equation}

The smallest $B$ serving for all cubes $I$ is called its $A_1$ ``norm": $\naone$.

\begin{thm}
\label{weakpv}
Let  $\phi$ be any function on $[1,\infty)$, $\phi(t)\ge t$. Let the operator $T$ has the property that for any $w\in A_1$
\begin{equation}
\label{weakaone}
\|T\|_{L^1(w)\rightarrow L^{1,\infty}(w)} \le c_1 \phi(\naone)\,.
\end{equation}
Then this operator satisfies
\begin{equation}
\label{weakatwo}
\|T\|_{L^2(w)\rightarrow L^{2,\infty}(w)} \le c_1 \phi(c_2\natwo)\,.
\end{equation}
\end{thm}

Let us combine this theorem with a remarkable result of Lerner--Ombrosio--P\'erez:

\begin{thm}
\label{lop1}
Let  $T$ be an arbitrary \cz operator. Then
\begin{equation}
\label{weakloga1}
\|T\|_{L^1(w)\rightarrow L^{1,\infty}(w)} \le c_1 \phi(\naone)\,,\,\,\text{where}\,\, \phi(t) =t\log (1+t)\,.
\end{equation}
\end{thm}

We obtain for any \cz operator

\begin{thm}
\label{pv1}
Let  $T$ be an arbitrary \cz operator. Then
\begin{equation}
\label{weakloga2}
\|T\|_{L^2(w)\rightarrow L^{2,\infty}(w)} \le c_1 \natwo\log (1+\natwo)\,.
\end{equation}
\end{thm}

Combining this with our first Theorem \ref{strongweak}
we get

\begin{thm}
\label{pv2}
Let  $T$ be an arbitrary \cz operator. Then
\begin{equation}
\label{strongloga2}
\|T\|_{L^2(w)\rightarrow L^{2}(w)} \le c_1 \natwo\log (1+\natwo)\,.
\end{equation}
\end{thm}

\noindent{\bf Remarks.} 1. This almost linear estimate should be replaced by a linear one. So far the obstacle is in Theorem \ref{lop1}.

\bigskip

The plan of the paper: we first prove Theorem \ref{strongweak}. It will take several sections. Then we prove Theorem \ref{weakpv}.
To prove Theorem \ref{strongweak} we need two more theorems. Let us introduce some notations: $T_{w^{-1}}$ denotes the operator $Tw^{-1}$ considered on $L^2(w^{-1}dx)$, $T_{w}$ denotes the operator $Tw$ considered on $L^2(wdx)$.  Notice a simple isometric formula
\begin{equation}
\label{isom}
\|T_{\ww}:L^2(\ww)\rightarrow L^2(w)\|=\nTtwo\,.
\end{equation}
For getting rid of inessential question we always think that our weight satisfies
$$
\lambda \le w \le L\,,
$$
for a very small positive $\lambda$ and a large $L$. But all estimates of norms we want to have {\it independent} of $\lambda$ and $L$ and dependent only on things like $\natwo$ and such. However, the assumption above allows to think that $T$ is already bounded wherever we want--the goal is to {\it find} the bound.

Looking through all cubes $I$ we denote by $K_{\chi}$ the smallest constant such that
\begin{equation}
\label{Kchi}
\|T_{\ww}\chi_I\|^2_{\ltwow}\le K_{\chi} \ww(I)\,,\,\,\text{and}\,\,\|T_w\chi_I\|^2_{\ltwowinv} \le K_{\chi}w(I)\,.
\end{equation}

\begin{thm}
\label{Kchithe}
$\sqrt{K_{\chi}}\le \nTtwoinf + \nTtwoinfinv\,.$
\end{thm}

\begin{thm}
\label{test}
$\nTtwo \le c_1 \natwo +c_2\sqrt{K_{\chi}}\,.$
\end{thm}

Obviously the combination of Theorems \ref{Kchithe}, \ref{test} gives our first main result, namely, Theorem \ref{strongweak}.

\section{The beginning of the proof of  Theorem \ref{test}}
\label{testbegin}

In what follows we use Nazarov-Treil-Volberg preprint \cite{NTVlost}.  We used the following {\bf fixed} notations:
\begin{equation}
\label{fixed1}
d\mu := \ww dx\,,\,\, d\nu:= w\,dx\,.
\end{equation}

Let $f\in L^2(\mu), g\in L^2(\nu)$ be two test functions. We can think without the loss of generality that they have the
compact support. Then let us think that their support is in $[\frac14,\frac34]^d$. Let $\mathcal{D}^{\mu}, \mathcal{D}^{\nu}$
be two dyadic lattices of
$\R$. We can think that they are both shifts of the same standard dyadic lattice $\mathcal{D}$, such that $[0,1]^d\in
\mathcal{D}$, and that
$\mathcal{D}^{\mu} =\mathcal{D}+\om_1,
\mathcal{D}^{\nu}=\mathcal{D}+\om_2$, where  vectors $\om_1,\om_2$ have all of their coordinates in $ [-\frac14,\frac14]$. We have a natural probability space of
pairs of such dyadic lattices:
$$
\Om :=\{(\om_1,\om_2) \in [-\frac14,\frac14]^{2n}\}
$$
provided with probability $\mathbb{P}$ which is equal to normalized Lebesgue measure on $[-\frac14,\frac14]^{2n}$.
We called these two independent dyadic lattices $\mathcal{D}^{\mu}, \mathcal{D}^{\nu}$ because they will be used to
decompose $f\in L^2(\mu), g\in L^2(\nu)$ correspondingly. This will be exactly the same type of decomposition as in the
``nonhomogeneous $T1$" theorems we met \cite{NTV1}-\cite{NTV5}. We use  the notion  of operators $\Delta_I^{\mu}, \Delta_J^{\nu}$.  (Notice that we will always keep the name $I$ for cubes from $\mathcal{D}^{\mu}$, and we always keep the name $J$ for cubes from $ \mathcal{D}^{\nu}$ !)  By this we mean the following. In $L^2(\mu)$ there is a subspace $\HH_I^{\mu}$ of function supported on $I$ and having constant values on each son of $I$ and such that the average value of a function with respect to $d\mu$ is zero. The orthogonal projections 
\begin{equation}
\label{ort1}
\Delta_I^{\mu}: L^2(\mu) \rightarrow \HHI
\end{equation}
are mutually orthogonal. Similarly, the orthogonal projections 
\begin{equation}
\label{ort1}
\Delta_J^{\nu}: L^2(\mu) \rightarrow \HHJ
\end{equation}
are mutually orthogonal.

In particular, we will be often using equalities
\begin{equation}
\label{ort2}
\sum_I\|\Delta_I^{\mu} f_I\|^2 = \|\sum_I\Delta_I^{\mu} f_I\|^2=\sup_{\psi\in L^2(\mu), \|\psi\|_{L^2(\mu)}=1}|\sum_I (\Delta_I^{\mu}f_I, \psi)_{\mu}|^2\,.
\end{equation}

We always use the {\bf fixed} notation

\begin{equation}
\label{fixed2}
\|\cdot\|_{\mu}:=\|\cdot\|_{L^2(\mu)}\,,\,\, \|\cdot\|_{\nu}:=\|\cdot\|_{L^2(\nu)}\,,\,\,(\cdot,\cdot)_{\mu}:=(\cdot,\cdot)_{L^2(\mu)}\,,\,\,(\cdot,\cdot)_{\nu}:=(\cdot,\cdot)_{L^2(\nu)}\,.
\end{equation}

Let us rewrite \eqref{ort2} (with the change of $\mu$ to $\nu$ and of course $I$ to $J$) as another equality which we will be using often

\begin{equation}
\label{ort3}
\sum_J\|\Delta_J^{\nu} f_J\|^2 = \|\sum_J\Delta_J^{\nu} f_J\|^2=\sup_{\psi\in L^2(\nu), \|\psi\|_{\nu}=1}|\sum_J (f_J, \Delta_J^{\nu}(\psi))_{\nu}|^2\,.
\end{equation}

One piece of notation:
Given a cubes $I\subset \hat{I}$ we introduce
$$
\PPI (\chi_{\hat{I}\setminus I}d\mu) := \int_{\hat{I}\setminus I} \frac{\ell(I)^{\eps}}{(\ell(I) +|c(I)-x|)^{d+\eps}} \,d\mu(x)\,.
$$
Here $c(I)$ denotes the center of the cube $I$.

Here is our first lemma.

\begin{lm}
\label{pivotal1}
Let $d\mu=w^{-1}dx, d\nu=wdx$, $w\in A_2$. Then for any cube $I\in D^{\mu}$ and any collection of disjoint open cubes $I_{\al}$, $I_{\al}\subset I$, we have
\begin{equation}
\label{PivotalCond}
\sum_{\al} [\PP_{I_{\al}}(\chi_{I\setminus I_{\al}}d\mu)]^2\nu(I_{\al}) \le K \mu(I)\,,
\end{equation}
with $K=c\natwo^2$.
\end{lm}

\begin{proof}
It is easy to see that
$$
[\PP_{I_{\al}}(\chi_{I\setminus I_{\al}}d\mu) \le c\,\inf_{x\in I_{\al}} M(\chi_I w^{-1})(x)\,.
$$
So
$$
\sum_{\al} [\PP_{I_{\al}}(\chi_{I\setminus I_{\al}}d\mu)]^2\nu(I_{\al}) \le c\,\int_I [M(\chi_I w^{-1}]^2(x)\,dw(x)\,.
$$
The last quantity is bounded by $c\,\natwo^2\,w^{-1}(I)$ by Buckley's theorem, see \cite{Buck1}.
We  are done.
\end{proof}

\bigskip

\noindent{\bf Introducing pivotal constant $K$.} Let us denote by $K$ the smallest possible quantity in the right hand side of \eqref{PivotalCond}. We call this constant the {\it pivotal} constant.
Let $\tilde{K} = 100\,K$.

\bigskip

Let us introduce the following notations:
Let $d\mu=w^{-1}dx, d\nu=wdx$, $w\in A_2$. Consider an arbitrary cube $\hat{I}\in \mathcal{D}^{\mu}$ and call a dyadic cube $I\in \mathcal{D}^{\mu}$  {\it stopping} if it is a maximal cube such that
\begin{equation}
\label{stop1}
[\PPI(\chi_{\hat{I}\setminus I}\,d\mu)]^2\nu(I) >\tilde{K}\, \mu(I)\,.
\end{equation}

Let us notice then

\begin{lm}
\label{pivotal2}
Let $d\mu=w^{-1}dx, d\nu=wdx$, $w\in A_2$. 
Let $\{I_{\al}\}$ denote stopping subcubes of $\hat{I}$. Then
\begin{equation}
\label{stop2}
\sum_{\al} \mu(I_{\al})\le \frac12 \mu(\hat{I})\,.
\end{equation}
\end{lm}

The proof is obvious from \eqref{stop1}, the choice of $K$ and from Lemma \ref{pivotal1}.
 We can introduce the pivotal constant if we change $\mu$ to $\nu$ but in the case $\mu=w^{-1}dx, \nu=wdx, w\in A_2$ it will be the same $c\,\natwo^2$ because of  the symmetry $\natwo=[w^{-1}]_{A_2}$.

\vspace{.2in}

The rest is devoted to the proof of the following result from \cite{NTVlost}. In this result measures $\mu, \nu$ are arbitrary, even the doubling property is not assumed. We use the notations
$$
[\mu,\nu]_{A_2}:=\sup_I\langle \mu\rangle_I\langle \nu\rangle_I\,,\,\,\text{where}\,\,\langle \mu\rangle_I:=\frac{\mu(I)}{|I|}\,.
$$

\begin{thm}
\label{pivotal3}
Let $\mu,\nu$ be two {\it arbitrary} measures on $\R$ satisfying the pivotal condition \eqref{PivotalCond} and its symmetric version with $\mu$ and $\nu$ exchanged, both with constant $K$. Let $T$ be an arbitrary \cz operator. Let the following test conditions be satisfied as well:
\begin{equation}
\label{Kchi}
\|T_{\mu}\chi_I\|^2_{\nu}\le K_{\chi} \mu(I)\,,\,\,\text{and}\,\,\|T_{\nu}\chi_I\|^2_{\mu} \le K_{\chi}\nu(I)\,.
\end{equation}
Then 
\begin{equation}
\label{KKchi}
\|T: L^2(\mu) \rightarrow L^2(\nu)\| \le c_0\sqrt{[\mu,\nu]_{A_2}}+  c_1\sqrt{K} + c_2 \sqrt{K_{\chi}}\,.
\end{equation}
\end{thm}

The proof of Theorem \ref{test} follows from Theorem \ref{pivotal3} immediately if we take into consideration that we proved that the pivotal constant for $\mu=w^{-1}dx, \nu=wdx, w\in A_2$ is $K=c\,\natwo^2$. Even though the above theorem is proved in \cite{NTVlost} we repeat here the proof with some modifications. We do this to slightly simplify \cite{NTVlost} and to make the roles of the contants involved in the proof completely transparent.

\section{The proof of Theorem \ref{pivotal3}. The start}
\label{pivotal3start}

Fix two test functions $f\in L^2(\mu), g\in L^2(\nu)$ and
consider $$
\Delta_I^{\mu}(f), \Delta_I^{\nu}(g)\,.
$$
Also, let $I_0^{\mu}$ denote the cube of $\mathcal{D}^{\mu}$ of side-length $1$ containing $\supp(f)$, the same about
$I_0^{\nu}$ changing $f$ to $g$ and $\mu$ to $\nu$.
$$
\Lambda^{\mu}(f) := (\int_{I_0^{\mu}} f\,d\mu)\,\chi_{I_0^{\mu}},\,\Lambda^{\nu}(g) := (\int_{I_0^{\nu}}
g\,d\nu)\,\chi_{I_0^{\nu}}\,.
$$

It is easy to see that functions $\Lambda^{\mu}(f), \Delta_I^{\mu}(f),  I\in \mathcal{D}^{\mu}$ are all pairwise orthogonal 
with respect to the scalar product $(\cdot,\cdot)_{\mu}$ of $L^2(\mu)$. The same is true for $\Lambda^{\nu}(f),
\Delta_I^{\nu}(f),  I\in \mathcal{D}^{\nu}$ with respect to the scalar product $(\cdot,\cdot)_{\nu}$ of $L^2(\nu)$.
 Thus,

\begin{equation}
\label{decompmu}
f = \Lambda^{\mu}(f) + \sum_{I\in\mathcal{D}^{\mu}, I\subset I_0^{\mu}}\Delta_I^{\mu}(f),\, \,\|f\|^2_{\mu} =
\|\Lambda^{\mu}(f)\|_{\mu}^2 + \sum_{I\in\mathcal{D}^{\mu}, I\subset I_0^{\mu}}\|\Delta_I^{\mu}(f)\|_{\mu}^2\,.
\end{equation}
Similarly,

\begin{equation}
\label{decompnu}
g = \Lambda^{\nu}(g) + \sum_{I\in\mathcal{D}^{\mu}, I\subset I_0^{\mu}}\Delta_I^{\nu}(g),\,\, \|g\|^2_{\nu} =
\|\Lambda^{\nu}(g)\|_{\nu}^2 + \sum_{I\in\mathcal{D}^{\nu}, I\subset I_0^{\nu}}\|\Delta_I^{\nu}(g)\|_{\nu}^2\,.
\end{equation}

These decompositions and the assumption \eqref{KKchi} imply in a very easy fashion that we can consider only
the case 
\begin{equation}
\label{Lambdazero}
\Lambda^{\mu}(f) = 0,\, \Lambda^{\nu}(g) =0\,. 
\end{equation}
In fact, $(T_{\mu}f, g)_{\nu} = (T_{\mu}f-\Lambda^{\mu}(f), g)_{\nu} + (\int_{I_0^{\mu}} f\,d\mu)(T_{\mu}(\chi_{I_0^{\mu}}),
g)_{\nu}$, and the second term is bounded by $\sqrt{K_{\chi}} \|f\|_{\mu}\|g\|_{\nu}$ trivially by \eqref{Kchi}. Using
\eqref{Kchi} one can get rid of $\Lambda^{\nu}(g)$ as well. 

So we always work under the assumption \eqref{Lambdazero}. We pay the constant $c\sqrt{K_{\chi}}$ to use this assumption.
Now, for simplicity, we think that $f,g$ are real valued. The proof will consist of cutting the sum below into several subsums (there seems to be at least seven of them) and using the cancellations separately in those sums:

$$
(T_{\mu} f, g)_{\nu} = \sum_{I\in \mathcal{D}^{\mu}, J\in \mathcal{D}^{\nu}}  (T_{\mu}\Delta_I^{\mu}(f),\Delta_J^{\nu}(g))_{\nu}\,.
$$

\subsection{Bad and good parts of $f$ and $g$}
\label{Badandgood}

We use ``good-bad" decomposition of test functions $f,g$ exactly as this has been done in \cite{NTV1}, \cite{NTV3}--\cite{NTV5}.
Consider two fixed lattices $\mathcal{D}^{\mu}, \mathcal{D}^{\nu}$ (so we fixed a point in $\Om$, see the notations above).
Fix forever $\delta$ as follows:
\begin{equation}
\label{delta}
\delta=\frac{\eps}{2(n+\eps)}\,.
\end{equation}

We call the cube $I\in \mathcal{D}^{\mu}$ bad if there exists $J\in\mathcal{D}^{\nu}$ such that
\begin{equation}
\label{14}
|J|\geq |I|,\,\,\, \dist (e(J), I) < \ell(J)^{1-\delta} \ell(I)^{\delta}\,.
\end{equation}
Here $e(J)$ is the union of boundaries of all sons of $J$. Similarly one defines bad cubes $J\in \mathcal{D}^{\nu}$.

\vspace{.2in}

\noindent{\bf Definition.}
We fix a large integer $r$ to be chosen later, and we say that $I\in \mathcal{D}^{\mu}$ is {\it essentially
bad} if there exists $J\in\mathcal{D}^{\nu}$ satisfying \eqref{14} such that it is much longer than $I$,
namely,
$\ell(J) \geq 2^r\,\ell(I)$.

If the cube is not essentially bad, it is called {\it good}.

Now
\begin{equation}
\label{essbad1}
f = f_{bad} + f_{good},\,\,\, f_{bad}:= \sum_{I\in\mathcal{D}^{\mu},\,I\,\text{is essentially bad}}\Delta_I^{\mu}f\,.
\end{equation}
The same type of decomposition is used for $g$:
\begin{equation}
\label{essbad2}
g = g_{bad} + g_{good},\,\,\, g_{bad}:= \sum_{J\in\mathcal{D}^{\nu},\,J\,\text{is essentially bad}}\Delta_I^{\nu}g\,.
\end{equation}

\subsection{Estimates on good functions}
\label{goodTwoWeightgood}

We refer the reader to \cite{NTV1}, \cite{NTV3}--\cite{NTV5} for the  detailed explanation 
that it is enough to estimate $|(T_{\mu} f_{good}g_{good})_{\nu}|$. However, here we also give an explanation for the sake of completeness.

\begin{equation}
\label{threeterms}
(T_{\mu} f, g)_{\nu} = (T_{\mu} f_{good}, g_{good})_{\nu} + (T_{\mu} f_{bad}, g_{good})_{\nu} + (T_{\mu} f,
g_{bad})_{\nu}\,.
\end{equation}
We repeat here sketchingly the reasoning of \cite{NTV1}, \cite{NTV3}--\cite{NTV5}. In \cite{NTV1}, \cite{NTV3}--\cite{NTV5} we proved the result that the mathematical expectation of $\|f_{bad}\|_{\mu}$,
$\|g_{bad}\|_{\nu}\|$ is small if $r$ is large. 
In fact, the proof of this fact is based on the observation that the {\it conditional} probability
\begin{equation}
\label{badprobTwoWeight}
\mathbb{P}\{(\om_1,\om_2)\in \Om : I\,\,\text{is essentially bad}\,| \,I\in\mathcal{D}^{\mu}\} \leq \tau(r)\rightarrow 0,
\,\,\, r\rightarrow
\infty\,.
\end{equation}

So we consider the following result as  already proved.

\begin{thm}
\label{badprobfTwoWeight}
 We consider the  decomposition of $f$ to bad and good part, and take a bad part of it for
every
$\om=(\om_1,\om_2)\in\Om$. Let $\E$ denote the expectation with respect to $(\Om,\mathbb{P})$. Then
\begin{equation}
\label{probfTwoWeight}
\E(\|f_{bad}\|_{\mu}) \leq \e(r)\|f\|_{\mu},\,\,\,\text{where}\,\,\, \e(r)\rightarrow 0, \,\,\, r\rightarrow \infty\,.
\end{equation}
The same with $g$:
\begin{equation}
\label{probgTwoWeight}
\E(\|g_{bad}\|_{\nu}) \leq \e(r)\|g\|_{\nu},\,\,\,\text{where}\,\,\, \e(r)\rightarrow 0, \,\,\, r\rightarrow \infty\,.
\end{equation}
\end{thm}

\vspace{.2in}

Coming back to \eqref{threeterms} we get
$$
|(T_{\mu} f, g)_{\nu}| \leq |(T_{\mu} f_{good}, g_{good})_{\nu}| + \|T_{\mu}\| \|f_{bad}\|_{\mu}\| g_{good}\|_{\nu} +
\|T_{\mu}\|\| f\|_{\mu} \|g_{bad}\|_{\nu}\leq
$$
$$
|(T_{\mu} f_{good}, g_{good})_{\nu}| + 2C\e(r)\|f\|_{\mu}\|g\|_{\nu},
$$
where $C$ temporarily denotes 
$\|T_{\mu}\|_{L^2(\mu)\rightarrow L^2(\nu)}$ (a priori finite, see our assumption $\lambda <w <L$ above).
Choosing $r$ to be such that $C\e(r) < \frac14$, choosing $f,g$ to make$|(T_{\mu}
f, g)_{\nu}|$ to almost attain  $C\|f\|_{\mu}\|g\|_{\nu}$, and taking the mathematical expectation,  we get
$$
\frac12 C  \|f\|_{\mu}\|g\|_{\nu} \leq \mathbb{E}|(T_{\mu} f_{good}, g_{good})_{\nu}|
$$ 
for these special $f,g$. If we manage to prove that {\it for all} $f,g$ 
\begin{equation}
\label{goodgoodTwoWeight}
|(T_{\mu} f_{good}, g_{good})_{\nu}| \leq c\,(\sqrt{K} +\sqrt{K_{\chi}}) \|f\|_{\mu}\|g\|_{\nu}\,\,\,\forall f\in L^2(\mu),
\forall g\in L^2(\nu),
\end{equation}
then we obtain
$$
\|T_{\mu}\|_{L^2(\mu)\rightarrow L^2(\nu)} = c\,(\sqrt{K} +\sqrt{K_{\chi}})\,,
$$
which finishes the proof of Theorem \ref{pivotal3}. 

The rest is devoted to the proof of \eqref{goodgoodTwoWeight}.

\section{First reduction of the estimate on good functions \eqref{goodgoodTwoWeight}. The diagonal part.}
\label{diagonal}

So let lattices $\mathcal{D}^{\mu},\mathcal{D}^{\nu}$ be fixed, and let $f,g$ be two good functions with respect to these
lattices.
Boundedness on characteristic functions declared in \eqref{KKchi} obviously implies
\begin{equation}
\label{obv}
|(T_{\mu} \Delta_I^{\mu} f, \Delta_J^{\nu}g)_{\nu}| \leq \sqrt{K_{\chi}}\|\Delta_I^{\mu} f\|_{\mu}\|\Delta_J^{\nu}g\|_{\nu}\,.
\end{equation}
In fact, in the left hand side we have $2^{d}\times 2^d$ terms enumerated by sons of $I$ and sons of $J$. Let $s_I, d_J$ be two such sons,
and $\Delta_I^{\mu} f= c_s$ on $s_I$ and $\Delta_J^{\nu}g=c_d$ on $d_J$. Obviously,
$$
|c_s|\le \frac{\|\Delta_I^{\mu}(f)\|_{\mu}}{\mu(s_I)^{1/2}}\,,\,\,|c_d|\le \frac{\|\Delta_J^{\nu}(g)\|_{\nu}}{\nu(d_J)^{1/2}}\,.
$$
Then
$$
|(T_{\mu} \chi_{s_I}\Delta_I^{\mu} f, \chi_{d_J}\Delta_J^{\nu}g)_{\nu}| \leq \frac{\|\Delta_I^{\mu}(f)\|_{\mu}}{\mu(s_I)^{1/2}} \frac{\|\Delta_J^{\nu}(g)\|_{\nu}}{\nu(d_J)^{1/2}}(T_\mu\chi_{s_I}, \chi_{d_J}) \le 
$$
$$
\sqrt{K_{\chi}}\frac{\|\Delta_I^{\mu}(f)\|_{\mu}}{\mu(s_I)^{1/2}}  \frac{\|\Delta_J^{\nu}(g)\|_{\nu}}{\nu(d_J)^{1/2}}   \mu(s_I)^{1/2}\nu(d_J)^{1/2}\leq
$$
$$
\sqrt{K_{\chi}}\|\Delta_I^{\mu}(f)\|_{\mu}\|\Delta_J^{\nu}(g)\|_{\nu}\,.
 $$

Therefore, in the sum $(T_{\mu}f,g)_{\nu} = \sum_{I\in \mathcal{D}^{\mu}, J\in \mathcal{D}^{\nu}}(T_{\mu}\Delta_I^{\mu},
\Delta_J^{\nu}g)_{\nu}$ the ``diagonal" part can be easily estimated. Namely,  by \eqref{obv} (below $r$ is the number involved in the
definition of good functions in the previous section, and we always have $I\in \mathcal{D}^{\mu}, J\in \mathcal{D}^{\nu}$
without mentioning this):
\begin{equation}
\label{diagonalTwoWeight}
\sum_{ 2^{-dr}|J| \leq |I| \leq 2^{dr} |J|, \dist(I,J) \leq\max(|I|,|J|)}|(T_{\mu}
\Delta_I^{\mu} f, \Delta_J^{\nu}g)_{\nu}| \leq \sqrt{K_{\chi}}\|f\|_{\mu}\|g\|_{\nu}\,.
\end{equation}

\section{Second reduction of the estimate on good functions \eqref{goodgoodTwoWeight}. A piece of long range interaction}
\label{LongRangeTwoWeight1}

Let us consider the sums 
\begin{equation}
\label{longrangesums1}
\Sigma_1:=\sum_{ 2^{-dr}|J| \leq |I|\leq |J|, \dist(I,J) \geq\ \ell(J)}|(T_{\mu}
\Delta_I^{\mu} f, \Delta_J^{\nu}g)_{\nu}| \,.
\end{equation}
\begin{equation}
\label{longrangesums2}
\Sigma_2:=\sum_{ 2^{-dr}|I| \leq |J|\leq |I|,
\dist(I,J) \geq \ell(I)}|(T_{\mu}
\Delta_I^{\mu} f, \Delta_J^{\nu}g)_{\nu}| \,.
\end{equation}

They can be estimated in a symmetric fashion. So we will only deal with the first one.

\begin{lm}
\label{longrangelmTW}
Let $|I|\leq |J|$, $\dist (I,J) \geq \ell(J)$. Then
\begin{equation}
\label{yellow4}
|(T_{\mu}
\Delta_I^{\mu} f, \Delta_J^{\nu}g)_{\nu}|\leq A\,\frac{\ell(I)^{\eps}}{(\dist(I,J)
+\ell(I)+\ell(J))^{d+\eps}}\mu(I)^{1/2}\nu(J)^{1/2}\|\Delta_I^{\mu} f\|_{\mu}\|\Delta_J^{\nu}g\|_{\nu}\,. 
\end{equation}
\end{lm}

\begin{proof}
Let $c$ be the center of cube $I$. We use the fact that $\int \Delta_I^{\mu}f\,d\mu =0$ to write
$$
(T_{\mu}
\Delta_I^{\mu} f, \Delta_J^{\nu}g)_{\nu} = \int_I d\mu(t) \int_J d\nu(s) K(t,s)\Delta_I^{\mu} f(t)\Delta_J^{\nu}g(s)=
$$
$$
\int_I d\mu(t) \int_J d\nu(s) (K(t,s)-K(c,s))\Delta_I^{\mu} f(t)\Delta_J^{\nu}g(s)\,.
$$
Then one can easily see that
\begin{equation}
\label{kerneldifference}
|(T_{\mu}
\Delta_I^{\mu} f, \Delta_J^{\nu}g)_{\nu}| \leq A\, \int_I\int_J\frac{\ell(I)^{\eps}}{|t-s|^{d+\eps}}|\Delta_I^{\mu}
f(t)||\Delta_J^{\nu}g(s)|d\mu(t)d\nu(s)\,.
\end{equation}
Now we estimate the kernel $\frac{\ell(I)^{\eps}}{(\ell(J) +|t-s|)^{d+\eps}}\chi_I(t)\chi_J(s)\leq A\, \frac{\ell(I)^{\eps}}{(\dist(I,J)
+\ell(I)+\ell(J))^{d+\eps}}$ using that $|I|\leq |J|$, $\dist (I,J) \geq \ell(J)$. On the other hand
$$
\|\Delta_I^{\mu} f\|_{L^1(\mu)} \leq \mu(I)^{1/2} \|\Delta_I^{\mu} f\|_{\mu},\,\,
\|\Delta_J^{\nu} g\|_{L^1(\nu)} \leq \mu(J)^{1/2} \|\Delta_J^{\nu} g\|_{\nu}\,.
$$
And the lemma is proved.
\end{proof}

Let us notice that Lemma \ref{longrangelmTW} allows us to write the following estimate for the sum of \eqref{longrangesums1}
(as usual $I\in \mathcal{D}^{\mu}, J\in \mathcal{D}^{\nu}$):

\begin{equation}
\label{lr1}
\Sigma_1 \leq \sum_{n=0}^{\infty}2^{-n\eps}\sum_{I,J: \ell(I)=2^{-n}\ell(J)}\frac{\ell(J)^{\eps}}{(\dist(I,J)
+\ell(I)+\ell(J))^{d+\eps}}\mu(I)^{1/2}\nu(J)^{1/2}\|\Delta_I^{\mu} f\|_{\mu}\|\Delta_J^{\nu}g\|_{\nu}\,.
\end{equation}
Or
\begin{equation}
\label{lr2}
\Sigma_1\leq \sum_{n=0}^{\infty}2^{-n\eps}\sum_{k\in\mathbb{Z}}\sum_{I,J:
\ell(I)=2^{-n+k}, \ell(J)=2^k}\frac{2^{k\eps}}{(\dist(I,J) + 2^k)^{d+\eps}}\mu(I)^{1/2}\nu(J)^{1/2}\|\Delta_I^{\mu}
f\|_{\mu}\|\Delta_J^{\nu}g\|_{\nu}\,.
\end{equation}
To estimate ``the $n,k$" slice
$$
\Sigma_{n,k} := \sum_{I,J:
\ell(I)=2^{-n+k}, \ell(J)=2^k}\frac{2^{k\eps}}{(\dist(I,J) + 2^k)^{d+\eps}}\mu(I)^{1/2}\nu(J)^{1/2}\|\Delta_I^{\mu}
f\|_{\mu}\|\Delta_J^{\nu}g\|_{\nu}
$$ 

let us introduce the notations.
$$
\f(t) = \sum_{I\in\mathcal{D}^{\mu}, \ell(I) =2^{-n+k}}\frac{\|\Delta_I^{\mu}
f\|_{\mu}}{\mu(I)^{1/2}}\chi_I(t),\,\,\psi(s) = \sum_{J\in\mathcal{D}^{\nu}, \ell(I) =2^{k}}\frac{\|\Delta_J^{\nu}
g\|_{\nu}}{\nu(J)^{1/2}}\chi_J(s)\,.
$$
Also 
$$
K_y(t,s) := \frac{y^{\eps}}{(y + |t-s|)^{d+\eps}},\,\, y > 0, \,\, t,s \in \mathbb{R}\,.
$$
Then 
\begin{equation}
\label{lr3}
\Sigma_{n,k} \leq \int_{\mathbb{R}}d\mu(t)\int_{\mathbb{R}}d\nu(s) K_{2^k}(t,s) \f(t)\psi(s)\,.
\end{equation}

\begin{lm}
\label{PoissonTW}
The integral operator $f\rightarrow \int K_y (t,s) \f(t)\,d\mu(t)$ is bounded from $L^2(\mu)$ to $L^2(\nu)$ if
$[\mu,\nu]_{A_2}<\infty$ (recall that this quantity is equal to $\sup_{I\subset\mathbb{R}}\langle\mu\rangle_I\langle\nu\rangle_I$). Its norm is bounded by
$c\,[\mu,\nu]_{A_2}^{1/2}$.
\end{lm}

Let us postpone the proof of this lemma, and let us finish the estimate of $\Sigma_1$ using it. First of all the lemma
gives the following estimate:
$$
\Sigma_{n,k} \leq c\,[\mu,\nu]_{A_2}^{1/2} \|\f\|_{\mu}\|\psi\|_{\nu} = c\,[\mu,\nu]_{A_2}^{1/2}(\sum_{I\in\mathcal{D}^{\mu},\, \ell(I) =2^{-n+k}}\|\Delta_I^{\mu}
f\|_{\mu}^2)^{1/2}(\sum_{J\in\mathcal{D}^{\nu}, \,\ell(J) =2^{k}}\|\Delta_J^{\nu}
g\|_{\nu}^2)^{1/2}\,.
$$
By Cauchy inequality
$$
\sum_k\Sigma_{n,k}\leq\sum_k (\sum_{J\in\mathcal{D}^{\nu},\, \ell(J) =2^{k}}\|\Delta_J^{\nu}
g\|_{\nu}^2)^{1/2}(\sum_{I\in\mathcal{D}^{\mu},\, \ell(I) =2^{-n+k}}\|\Delta_I^{\mu}
f\|_{\mu}^2)^{1/2} \leq 
$$
$$
(\sum_{J\in\mathcal{D}^{\nu}}\|\Delta_J^{\nu}
g\|_{\nu}^2)^{1/2}(\sum_{I\in\mathcal{D}^{\mu}}\|\Delta_I^{\mu}
f\|_{\mu}^2)^{1/2} \leq \|f\|_{\mu}\|g\|_{\nu}
$$
by \eqref{decompmu}.
Then \eqref{lr2} gives $\Sigma_1 \leq \sum_{n=0}^{\infty} 2^{-n} \sum_k \Sigma_{n,k}$, and so
$$
\Sigma_1\leq c\,[\mu,\nu]_{A_2}^{1/2}\sum_{n=0}^{\infty} 2^{-n}\|f\|_{\mu}\|g\|_{\nu} =c\,[\mu,\nu]_{A_2}^{1/2}\|f\|_{\mu}\|g\|_{\nu},
$$
and our first piece of long range interaction sum $\Sigma_1$ is finally estimated.

\vspace{.2in}

{\bf Proof of Lemma \ref{PoissonTW}}

\vspace{.2in}

Let us consider several other averaging operators. One of them is
$$
I\f(s) := \int \chi_{[-\frac12,\frac12]^d}(s-t)\f(t)\,d\mu(t)\,.
$$
Another is as follows: let $G$ be all cubes $\ell_k$, $k=(k_1,...,k_d)$ of the type $
[2k_1, 2k_1+2]\times\dots\times[2k_d, 2k_d+2]$, $k_i\in\mathbb{Z}$.
Consider
$$
A_G\f(s) := \sum_{k}\chi_{\ell_k}(s) \frac1{|\ell_k|}\int_{\ell_k}\f\,d\mu\,.
$$
Consider also shifted grid $G(x)= G+x, x\in [0,2)^d$, and corresponding $A_{G(x)}$.

Notice that
\begin{equation}
\label{IA}
I\f(s) \leq a\, \int_0^2 A_{G(x)}\f(s)\,dx\,.
\end{equation}
In fact, consider $[0,2]^d, \frac1{2^d}dx$ as an obvious probability space of all grids $G(x)$. Then 
it is easy to see that for every $s$ the unit cube $[s-\frac12, s+\frac12]^d$ is (with probability $c_d>0$) 
a subcube of one of the cubes of $G(x)$.
Then the above inequality becomes  obvious.

On the other hand, the norm of operator  $A_G$ as an operator from $L^2(\mu)$ to $L^2(\nu)$ is bounded by
$c\,[\mu,\nu]_{A_2}^{1/2}$. In fact, if  $\ell_k =
[2k_1, 2k_1+2]\times\dots\times[2k_d, 2k_d+2]$, then
$$
\|A_G\f\|^2_{\nu} \leq  \sum_k (\int_{\ell_k}|\f|\,d\mu)^2\nu(\ell_k) \leq \sum_k (\int_{\ell_k}|\f|^2\,d\mu)\nu(\ell_k)
\mu(\ell_k) \leq 
$$
$$
c\,[\mu,\nu]_{A_2}\,\sum_k \int_{\ell_k}|\f|^2\,d\mu= c\,[\mu,\nu]_{A_2}\,\|f\|_{\mu}^2\,.
$$
The same, of course, can be said about $\|A_{G(x)}\f\|^2_{\nu}$.
Then \eqref{IA} implies that the norm of averaging operator $I$ from $L^2(\mu)$ to $L^2(\nu)$ is bounded by $c\,[\mu,\nu]_{A_2}^{1/2}$.
Let us call by $I_r$ the operator of the same type as $I$, but the convolution now will be with the normalized
characteristic function of the cube $[-r,r]$:
$$
I_r\f(s) := \frac1{(2r)^d}\int \chi_{[-r,r]^d}(s-t)\f(t)\,d\mu(t)\,.
$$
It is obvious that the reasoning above can be repeated without any change and we get
\begin{equation}
\label{Ir}
\|I_r\f\|^2_{\nu} \leq c\,[\mu,\nu]_{A_2}^{1/2}\,\|f\|_{\mu}^2\,.
\end{equation}

To finish with the operator given by  $f\rightarrow \int K_y (t,s) \f(t)\,d\mu(t)$ as an operator from $L^2(\mu)$ to
$L^2(\nu)$, let us notice that (and this is a standard inequality for the Poisson-type kernels)
$$
\int K_y (t,s) |\f(t)|\,d\mu(t)\leq c\, \sum_{k=0}^{\infty}2^{-{k\eps}}(I_{y\cdot 2^k}|\f|)(s)\,.
$$
Now Lemma \ref{PoissonTW} follows immediately from \eqref{Ir} and the last inequality.

\section{ The rest of the long range
interaction}
\label{LongRangeTwoWeight2}

As always all $I$'s below are in $\mathcal{D}^{\mu}$, all $J$'s below are in $\mathcal{D}^{\nu}$.
Consider now the following two sums.

\begin{equation}
\label{longrangesums3}
\sigma_1:=\sum_{  |I|< 2^{-dr}|J|, I\cap J =\emptyset}|(T_{\mu}
\Delta_I^{\mu} f, \Delta_J^{\nu}g)_{\nu}| \,.
\end{equation}
\begin{equation}
\label{longrangesums4}
\sigma_2:=\sum_{  |J|< 2^{-dr}|I|,
I\cap J =\emptyset}|(T_{\mu}
\Delta_I^{\mu} f, \Delta_J^{\nu}g)_{\nu}| \,.
\end{equation}

They can be estimated in a symmetric fashion. So we will only deal with the first one.

Notice that $f,g$ are good functions. These means, in particular, that $I,J$, which we meet in \eqref{longrangesums3} satisfy
\begin{equation}
\label{14again}
\dist (I, \pd J) \geq \ell(J)^{1-\delta}\ell(I)^{\delta}\,,
\end{equation}
where $\delta$ was introduced in \eqref{delta}.
This is just \eqref{14} for disjoint $I,J$ with $I$ {\it not} essentially bad (see the definition at the
beginning of Subsection \ref{Badandgood}).

\begin{lm}
\label{longrangelmTW2}
Let $I,J$ be disjoint, $\ell(I)< 2^{-r}\ell(J)$,  and satisfy \eqref{14again}. Then
\begin{equation}
\label{yellow4again}
|(T_{\mu}
\Delta_I^{\mu} f, \Delta_J^{\nu}g)_{\nu}|\leq A\,\frac{\ell(I)^{\eps/2}\ell(J)^{\eps/2}}{(\dist(I,J)
+\ell(I) +\ell(J))^{d+\eps}}\mu(I)^{1/2}\nu(J)^{1/2}\|\Delta_I^{\mu} f\|_{\mu}\|\Delta_J^{\nu}g\|_{\nu}\,. 
\end{equation}
\end{lm}

\begin{proof}
If $\dist (I, J) \geq \ell(J)$, this has been already proved in Lemma \ref{longrangelmTW}.  And actually it was proved there with a better numerator: $\ell(I)^{\eps}$. So let $\dist (I, J) \leq \ell(J)$, $I,J$
being disjoint. Repeating \eqref{kerneldifference} one gets
$$
|(T_{\mu}
\Delta_I^{\mu} f, \Delta_J^{\nu}g)_{\nu}| \leq A\, \int_I\int_J\frac{\ell(I)^{\eps}}{|t-s|^{d+\eps}}|\Delta_I^{\mu}
f(t)||\Delta_J^{\nu}g(s)|d\mu(t)d\nu(s)\,.
$$
Now we estimate the kernel: $\frac{\ell(I)^{\eps}}{|t-s|^{d+\eps}}\chi_I(t)\chi_J(s)\leq A\, \frac{\ell(I)^{\eps}}{\dist(I,\pd J)^{d+\eps}}$. 
Therefore,
\begin{equation}
\label{yellow8}
|(T_{\mu}
\Delta_I^{\mu} f, \Delta_J^{\nu}g)_{\nu}| \leq A\, \frac{\ell(I)^{\eps}}{\dist(I,\pd J)^{d+\eps}}\mu(I)^{1/2}\nu(J)^{1/2}\|\Delta_I^{\mu}
f\|_{\mu}\|\Delta_J^{\nu}g\|_{\nu}\,. 
\end{equation}
We use \eqref{14again} to write
$$
\frac{\ell(I)^{\eps}}{\dist(I,\pd J)^{d+\eps}} \leq  \frac{\ell(I)^{\eps}}{\ell(I)^{\delta (d+\eps)} \ell(J)^{(1-\delta)(d+\eps)}}=
\frac{\ell(I)^{\eps-\delta (d+\eps)}\ell(J)^{\eps}}{\ell(J)^{d+\eps} \ell(J)^{\eps-\delta (d+\eps)}}=
$$
$$
\bigg(\frac{\ell(I)}{\ell(J)}\bigg)^{\eps/2}\frac{\ell(J)^{\eps}}{(\ell(I)+\dist(I,J)+\ell(J))^{d+\eps}}
$$
because we assumed $\dist (I, J) \leq \ell(J)$ and $I$ is shorter than $J$. This inequality and \eqref{yellow8} finish the proof
of the lemma.
\end{proof}

Let us notice that Lemma \ref{longrangelmTW2} allows to write the following estimate for the sum $\sigma_1$ from
\eqref{longrangesums3}:

\begin{equation}
\label{lr10}
\sigma_1 \leq \sum_{n=0}^{\infty}2^{-n/2}\sum_{I,J: \ell(I)=2^{-n}\ell(J)}\frac{\ell(J)^{\eps}}{(\dist(I,J)
+\ell(I)+\ell(J))^{d+\eps}}\mu(I)^{1/2}\nu(J)^{1/2}\|\Delta_I^{\mu} f\|_{\mu}\|\Delta_J^{\nu}g\|_{\nu}\,.
\end{equation}
Or
\begin{equation}
\label{lr20}
\sigma_1\leq \sum_{n=0}^{\infty}2^{-n\eps/2}\sum_{k\in\mathbb{Z}}\sum_{I,J:
\ell(I)=2^{-n+k}, \ell(J)=2^k}\frac{2^{k\eps}}{(\dist(I,J) + 2^k)^{d+\eps}}\mu(I)^{1/2}\nu(J)^{1/2}\|\Delta_I^{\mu}
f\|_{\mu}\|\Delta_J^{\nu}g\|_{\nu}\,.
\end{equation}
To estimate ``the $n,k$" slice
$$
\sigma_{n,k} := \sum_{I,J:
\ell(I)	=2^{-n+k}, \ell(J)=2^k}\frac{2^{k\eps}}{(\dist(I,J) + 2^k)^{d+\eps}}\mu(I)^{1/2}\nu(J)^{1/2}\|\Delta_I^{\mu}
f\|_{\mu}\|\Delta_J^{\nu}g\|_{\nu}
$$ 
let us  use again the notations
$$
\f(t) = \sum_{I\in\mathcal{D}^{\mu}, \ell(I) =2^{-n+k}}\frac{\|\Delta_I^{\mu}
f\|_{\mu}}{\mu(I)^{1/2}}\chi_I(t),\,\,\psi(s) = \sum_{J\in\mathcal{D}^{\nu}, \ell(I) =2^{k}}\frac{\|\Delta_J^{\nu}
g\|_{\nu}}{\nu(J)^{1/2}}\chi_J(s)\,.
$$
Also 
$$
K_y(t,s) := \frac{y^{\eps}}{(y + |t-s|)^{d+\eps}},\,\, y > 0, \,\, t,s \in \mathbb{R}\,.
$$
Then 
\begin{equation}
\label{lr30}
\sigma_{n,k} \leq \int_{\mathbb{R}}d\mu(t)\int_{\mathbb{R}}d\nu(s) K_{2^k}(t,s) \f(t)\psi(s)\,.
\end{equation}

Lemma \ref{PoissonTW} now gives as before the estimate of $\sigma_1$. First of all the lemma
gives the following estimate:
$$
\sigma_{n,k} \leq c\, [\mu,\nu]_{A_2}^{1/2}\, \|\f\|_{\mu}\|\psi\|_{\nu} = c\, [\mu,\nu]_{A_2}^{1/2}\, (\sum_{I\in\mathcal{D}^{\mu},\, \ell(I) =2^{-n+k}}\|\Delta_I^{\mu}
f\|_{\mu}^2)^{1/2}(\sum_{J\in\mathcal{D}^{\nu}, \,\ell(J) =2^{k}}\|\Delta_J^{\nu}
g\|_{\nu}^2)^{1/2}\,.
$$
By Cauchy inequality
$$
\sum_k\sigma_{n,k}\leq\sum_k (\sum_{J\in\mathcal{D}^{\nu},\, \ell(J) =2^{k}}\|\Delta_J^{\nu}
g\|_{\nu}^2)^{1/2}(\sum_{I\in\mathcal{D}^{\mu},\, \ell(I) =2^{-n+k}}\|\Delta_I^{\mu}
f\|_{\mu}^2)^{1/2} \leq 
$$
$$
(\sum_{J\in\mathcal{D}^{\nu}}\|\Delta_J^{\nu}
g\|_{\nu}^2)^{1/2}(\sum_{I\in\mathcal{D}^{\mu}}\|\Delta_I^{\mu}
f\|_{\mu}^2)^{1/2} \leq \|f\|_{\mu}\|g\|_{\nu}
$$
by \eqref{decompmu}.
Then \eqref{lr20} gives $\sigma_1 \leq \sum_{n=0}^{\infty} 2^{-n\eps/2} \sum_k \sigma_{n,k}$, and so
$$
\sigma_1\leq c\, [\mu,\nu]_{A_2}^{1/2}\,\sum_{n=0}^{\infty} 2^{-n\eps/2}\|f\|_{\mu}\|g\|_{\nu} =c\, [\mu,\nu]_{A_2}^{1/2}\,\|f\|_{\mu}\|g\|_{\nu},
$$
and our long range interaction sum $\sigma_1$ is finally estimated.
Symmetric estimate holds for $\sigma_2$ from \eqref{longrangesums4}.

\vspace{.2in}

\noindent{\bf Conclusion:} if $f,g$ are good, then the sum of all terms $|(T_{\mu}
\Delta_I^{\mu} f, \Delta_J^{\nu}g)_{\nu}| $ such that either $\frac{\ell(I)}{\ell(J)}\in [2^{-r}, 2^r]$ or $I\cap J=\emptyset$
has the  estimate $c\, [\mu,\nu]_{A_2}^{1/2}\,\|f\|_{\mu}\|g\|_{\nu}$.

\section{ The short range interaction. Corona decomposition.}
\label{ShortRangeI}

As always all $I$'s below are in $\mathcal{D}^{\mu}$, all $J$'s below are in $\mathcal{D}^{\nu}$.

\noindent Let us consider the sums 
\begin{equation}
\label{srsums1}
\rho :=\sum_{  |I|< 2^{-dr}|J|, I\subset J, \,, \,I\,\text{is good}}(
\Delta_I^{\mu} f,T'_{\nu} \Delta_J^{\nu}g)_{\mu} \,.
\end{equation}
\begin{equation}
\label{srsums2}
\tau :=\sum_{  |J|< 2^{-dr}|I|, J\subset I, J\in \mathcal{D}_{\nu}\,, \,J\,\text{is good}}
(T_{\mu}\Delta_I^{\mu} f, \Delta_J^{\nu}g)_{\nu} \,.
\end{equation}
%\dist (J,e(I)) \geq |I|^{3/4}|J|^{1/4}}
%\tag
They can be estimated in a symmetric fashion. So we will only deal with, say, the second one.
It is very important that unlike the sums $\Sigma_i$, $\sigma_i$, this sum does not have absolute value on 
{\it each} term.

Consider each term of $\tau$ and split it to three terms.
To do this, let $I_i$ denote the half of $I$, which contains $J$. And $I_n$ is all other sons.
Let $\hat{I}$ denote an arbitrary super cube of $I_i$ in the same lattice: $\hat{I} \in \mathcal{D}^{\mu}$.

\bigskip

We write
$$
(T_{\mu}\Delta_I^{\mu} f, \Delta_J^{\nu}g)_{\nu} = (T_{\mu}(\chi_{I_n}\Delta_I^{\mu} f), \Delta_J^{\nu}g)_{\nu} +
(T_{\mu}(\chi_{I_i}\Delta_I^{\mu} f), \Delta_J^{\nu}g)_{\nu} =
$$
$$
(T_{\mu}(\chi_{I_n}\Delta_I^{\mu} f), \Delta_J^{\nu}g)_{\nu} + \langle \Delta_I^{\mu}f\rangle_{\mu, I_i}
(T_{\mu}(\chi_{\hat{I}}),\Delta_J^{\nu}g)_{\nu} - \langle \Delta_I^{\mu}f\rangle_{\mu, I_i}
(T_{\mu}(\chi_{\hat{I}\setminus I_i}),\Delta_J^{\nu}g)_{\nu}\,.
$$
Here $\langle \Delta_I^{\mu}f\rangle_{\mu, I_i}$ is the average of $\Delta_I^{\mu}f$ with respect to $\mu$ over $I_i$,
which is the same as value of this function on $I_i$ (by construction $\Delta_I^{\mu}f$ assumes on $I$ $2^d$, one on
$I_i$, others on $I_n$).

\vspace{.2in}

\noindent{\bf Definition.} We call them as follows: the first one is ``the neighbor-term",
the second one is ``the difficult term", the third one is ``the stopping term".

\bigskip

Notice that it may happen that $\hat{I}=I_i$. Then stopping term is zero. 

\subsection{The estimate of neighbor-terms}
\label{neighborterms}

We have the same estimate as in Lemma \ref{longrangelmTW2}:
$$
|(T_{\mu}(\chi_{I_n}\Delta_I^{\mu} f), \Delta_J^{\nu}g)_{\nu}|\leq 
$$
\begin{equation}
\label{nt}
A\,\frac{\ell(I)^{\eps/2}\ell(J)^{\eps/2}}{(\dist(I,J)
+\ell(I)+\ell(J)^{d+\eps}}\mu(I)^{1/2}\nu(J)^{1/2}\|\chi_{I_n}\Delta_I^{\mu} f\|_{\mu}\|\Delta_J^{\nu}g\|_{\nu}\,. 
\end{equation}
Of course, $\|\chi_{I_n}\Delta_I^{\mu} f\|_{\mu}\leq \|\Delta_I^{\mu} f\|_{\mu}$. So the estimate of the sum of absolute
values of neighbor-terms is exactly the same as the estimate of $\sigma_1$ in the preceding section.

\subsection{The estimate of stopping terms}
\label{stoppingterms}

We want to estimate
$$
|\langle \Delta_I^{\mu}f\rangle_{\mu, I_i}|
|(T_{\mu}(\chi_{\hat{I}\setminus I}),\Delta_J^{\nu}g)_{\nu}|\,.
$$
First of all, obviously
$$
|\langle \Delta_I^{\mu}f\rangle_{\mu, I_i}| \leq \frac{\|\Delta_I^{\mu} f\|_{\mu}}{\mu(I_i)^{1/2}}\,.
$$
Secondly,
$$
|(T_{\mu}(\chi_{\hat{I}\setminus I}),\Delta_J^{\nu}g)_{\nu}| = |(\chi_{\hat{I}\setminus I},T'_{\nu}\Delta_J^{\nu}g)_{\mu}|\leq
$$
$$
A\, \Bigl(\int_{\hat{I}\setminus I}d\mu(x)\frac{\ell(J)^{\eps}}{\dist(x,J)^{d+\eps}}\Bigr) \|\Delta_J^{\nu}g\|_{L^1(\nu)}\,.
$$ 
This is the usual trick with subtraction of the kernel, it uses the fact that $\int \Delta_J^{\nu}g\,d\nu=0$.
We continue by denoting the center of $I_i$ by $c(I_i)$. Consider two cases:
1) $\dist (x, J) \le 10 \ell(I)$, in this case (we use that $J$ is good)
$$
\frac{\ell(J)^{\eps}}{\dist(x,J)^{d+\eps}} \le  \frac{\ell(J)^{\eps}}{\ell(J)^{\delta (d+\eps)} \ell(I)^{(1-\delta)(d+\eps)}}=
\frac{\ell(J)^{\eps-\delta (d+\eps)}\ell(I)^{\eps}}{\ell(I)^{d+\eps} \ell(I)^{\eps-\delta (d+\eps)}}\le
$$
$$
 \bigg(\frac{\ell(J}{\ell(I)}\bigg)^{\eps/2}\frac{\ell(I)^{\eps}}{\dist(x, c(I_i))^{d+\eps}}\,;
$$
 2)  $\dist (x, J) \ge 10 \ell(I)$, in this case 
$$
\frac{\ell(J)^{\eps}}{\dist(x,J)^{d+\eps}} \le c\,\bigg(\frac{\ell(J}{\ell(I)}\bigg)^{\eps}\frac{\ell(I)^{\eps}}{\dist(x, c(I_i))^{d+\eps}}\,.
$$

 We continue, using the definition above,
$$
|(T_{\mu}(\chi_{\hat{I}\setminus I}),\Delta_J^{\nu}g)_{\nu}| \leq  A\,\nu(J)^{1/2}\|\Delta_J^{\nu}g\|_{\nu}\int_{\hat{I}\setminus
I}\bigg(\frac{\ell(J)}{\ell(I)}\bigg)^{\eps/2}\frac{\ell(J)^{\eps}}{\dist(x,c(I_i))^{d+\eps}}\,d\mu(x)\leq
$$
$$
 \le A\,\nu(J)^{1/2}\|\Delta_J^{\nu}g\|_{\nu}\Bigl(\frac{\ell(J)}{\ell(I)}\Bigr)^{\eps/2} \PP_{I_i}(\chi_{\hat{I}\setminus
I} \,d\mu)\,.
$$

Thus 
\begin{equation}
\label{yellow9}
|(T_{\mu}(\chi_{\hat{I}\setminus
I}),\Delta_J^{\nu}g)_{\nu}|\leq A\,\nu(J)^{1/2}\|\Delta_J^{\nu}g\|_{\nu}\Bigl(\frac{\ell(J)}{\ell(I)}\Bigr)^{\eps/2}
\PP_{I_i}(\chi_{\hat{I}\setminus I} \,d\mu)\,.
\end{equation}

We now get the estimate of the stopping term:
\begin{equation}
\label{yellow10}
|\langle \Delta_I^{\mu}f\rangle_{\mu, I_i}|
|(T_{\mu}(\chi_{\hat{I}\setminus I}),\Delta_J^{\nu}g)_{\nu}|\leq A\,
\Bigl(\frac{\nu(J)}{\mu(I_i)}\Bigr)^{1/2}\Bigl(\frac{\ell(J)}{\ell(I)}\Bigr)^{\eps/2}\PP_{I_i}(\chi_{\hat{I}\setminus
I_i} \,d\mu)\|\Delta_J^{\nu}g\|_{\nu}\|\Delta_I^{\mu}f\|_{\mu}\,.
\end{equation}

\subsection{Pivotal property}
\label{pivotalproperty}
Let $I\in \mathcal{D}_{\mu}$. Let $\{I_{\al}\}$ be a finite family of {\it disjoint} subcubes of $I$ belonging to the same lattice. We have called (at the beginning of the paper) the following property {\it pivotal property}:

\begin{equation}
\label{PIVOTAL}
\sum_{\al} [\PP_{I_{\al}}(\chi_{I\setminus I_{\al}}d\mu)]^2\nu(I_{\al}) \leq K\,\mu(I)\,.
\end{equation}

Recall that in the case 
$$
\mu=w^{-1}dx\,,\, \nu=wdx\,,\, w\in A_2
$$
the property \eqref{PIVOTAL} is satisfied with
$$
K=c\, \natwo^2\,.
$$

\subsection{The choice of stopping cubes}
\label{stoppingchoice}

Fix a cube $\hat{I}\in \mathcal{D}^{\mu}$.
%Let us call its {\it subcube} $I\in \mathcal{D}^{\mu}$ a {\it stopping cube} if it is the last one (by %going to the
%smaller ones by inclusion) such that for its both sons $I_1, I_2$
%\begin{equation}
%\label{yellow11}
%\Bigl[P_{I_i}(\chi_{\hat{I}\setminus I_i}\,d\mu)\Bigr]^2 \nu(I_i) \leq K\, \mu(I_i),\,\, i=1,2\,.
%\end{equation}
%\tag
Let us call its {\it subcubes} $I\in \mathcal{D}^{\mu}$ a {\it stopping cubes} if it is the first one (by going from bigger ones to the
smaller ones by inclusion) such that 
\begin{equation}
\label{yellow11}
\Bigl[\PP_{I}(\chi_{\hat{I}\setminus I}\,d\mu)\Bigr]^2 \nu(I) \geq 100\,K\, \mu(I),\,\, i=1,2\,,
\end{equation}
where $K$ is the constant from \eqref{PIVOTAL}.

\bigskip

Here is the place, where we use the pivotal property \eqref{PIVOTAL}:

\begin{thm}
\label{yellowS}
If $\mu, \nu$ are arbitrary positive measures such that \eqref{PIVOTAL} is satisfied, then for every $\hat{I}\in
\mathcal{D}^{\mu}$
\begin{equation}
\label{yellowCarl}
\sum_{I\in \mathcal{D}^{\mu},\, I\subset \hat{I}, \,I \text{is maximal stopping}}\mu(I) \leq \frac12 \mu(\hat{I})\,,
\end{equation}
\end{thm}

%%%%%%%%%%%%%%%%%%%%%%%%%%%%
\begin{proof}
In fact, let $\{I_{\al}\}$ be a family of maximal stopping cubes inside $\hat{I}$ according to stopping criteria just introduced in \eqref{yellow11}. Then
$$
\mu(I_{\al})\leq \frac1{100\,K} \Bigl[\PP_{I_{\al}}(\chi_{\hat{I}\setminus I_{\al}}\,d\mu)\Bigr]^2 \nu(I_{\al})\,.
$$
cubes $\{I_{\al}\}$  are disjoint subcubes of $\hat{I}$, and so \eqref{PIVOTAL} is used now:
$$
\sum_{\al} \mu(I_{\al})\leq \frac1{100\,K} \sum_{\al} \Bigl[\PP_{I_{\al}}(\chi_{\hat{I}\setminus I_{\al}}\,d\mu)\Bigr]^2 \nu(I_{\al})
\leq
$$
$$
\frac{K}{100\,K}\mu(\hat{I})\leq \frac12\mu(\hat{I})\,.
$$
\end{proof}
%%%%%%%%%%%%%%%%%%%%%%%%%%%%%%
\noindent{\bf Definitions.} 1. For any dyadic cube $I$, $F(I)$ will denote its father. 

\noindent 2. The  tree distance between the dyadic cubes of the same lattice will be denoted by $t(I_1,I_2)$. Of course $t(I, F(I))=1$.

\noindent 3. Stopping cubes  of the same lattice  will also form a tree. We will call it $\mathcal{S}$. The tree distance inside 
$\mathcal{S}$ will be denoted by $r(S_1,S_2)$. 
Of course
\begin{equation}
\label{trivialrt}
r(S_1,S_2) \leq t(S_1,S_2)\,.
\end{equation}

\subsection{Stopping tree}
\label{stoppingtree}

In Section \ref{ShortRangeI} we introduced the sum, which we are left to estimate:

\begin{equation}
\label{srsums20}
\tau :=\sum_{  |J|< 2^{-dr}|I|, J\subset I, J\in \mathcal{D}_{\nu}, J\,\text{is good}}
(T_{\mu}\Delta_I^{\mu} f, \Delta_J^{\nu}g)_{\nu} \,.
\end{equation}
%\dist (J,e(I)) \geq |I|^{3/4}|J|^{1/4}}(T_{\mu}
%\tag

Each term of $\tau$ was decomposed into three terms.
We recall: let $I_i$ denote the son of $I$, which contains $J$. And $I_n$ is the union of other sons.
Let $\hat{I}$ denote an arbitrary supercube of $I_i$ in the same lattice: $\hat{I} \in \mathcal{D}^{\mu}$.

%%%%%%%%%%%%%%%%%%%%%%%%%%%%%%%%%%%%%%%%%%
%%%%%%%%%%%%%%%%%%%%%%%%%%%%%%%%%%%%%%%%%%
%\tag Before it was Let $\hat{I}$ denote an arbitrary supercube of $I$ in the same lattice
%The difference is huge. Now the first bad cube--$I_i$---will be stopping, unlike the last good-
%--$I$---as before in \cite{NTV7}.
%%%%%%%%%%%%%%%%%%%%%%%%%%%%%%%%%%%%%%%%%%

For a given $I\in \mathcal{D}_{\mu}$, $J\subset I, J\in \mathcal{D}_{\nu}$, $J$ good, we write down the following splitting
$$
(T_{\mu}\Delta_I^{\mu} f, \Delta_J^{\nu}g)_{\nu} = (T_{\mu}(\chi_{I_n}\Delta_I^{\mu} f), \Delta_J^{\nu}g)_{\nu} +
(T_{\mu}(\chi_{I_i}\Delta_I^{\mu} f), \Delta_J^{\nu}g)_{\nu} =
$$
\begin{equation}
\label{splittingto3terms}
(T_{\mu}(\chi_{I_n}\Delta_I^{\mu} f), \Delta_J^{\nu}g)_{\nu} + \langle \Delta_I^{\mu}f\rangle_{\mu, I_i}
(T_{\mu}(\chi_{\hat{I}}),\Delta_J^{\nu}g)_{\nu} - \langle \Delta_I^{\mu}f\rangle_{\mu, I_i}
(T_{\mu}(\chi_{\hat{I}\setminus I_i}),\Delta_J^{\nu}g)_{\nu}\,.
\end{equation}
Here $\langle \Delta_I^{\mu}f\rangle_{\mu, I_i}$ is the average of $\Delta_I^{\mu}f$ with respect to $\mu$ over $I_i$,
which is the same as value of this function on $I_i$ (by construction $\Delta_I^{\mu}f$ assumes on $I$  exactly $2^d$ values, one on
$I_i$, others on $I_n$).

We called them as follows: the first one is ``the neighbor-term",
the second one is ``the difficult term", the third one is ``the stopping term".

In what follows it is convenient to think that we consider our problem on the circle $\mathbb{T}$ rather than on the line.
We want to explain how to choose $\hat{I}$ in a stopping terms above.

\vspace{.2in}

\noindent{\bf Construction of the stopping tree $\mathcal{S}$}.
We choose first $\hat{I}=I_0$, where $I_0$ is the unit cube of the lattice $\mathcal{D}^{\mu}$ which contains the support of $f$.
The choose its maximal stopping subcubes $\{I\}$. Just use the criterion \eqref{yellow11} from Subsection
\ref{stoppingchoice}. Call each of these $I$'s by the name $\hat{S}$. In each $\hat{S}$ again find  its maximal stopping
subcubes $\{S\}$. Et cetera... . All cubes, which were thus built, we call ``stopping cubes". They have their
generation. Stopping cubes, as a rule, will be denoted by symbols with ``hats". 

To explain the choice of $\hat{I}$ in a stopping terms above we need the notations. If $R$ is a cube in $\R$ we call $Q_R$ the cube in one more dimension built on $R$ as on its base. Sometimes we call $L(R)$ its {\it upper} face.

\vspace{.2in}

\noindent{\bf Notations.} If $\hat{S}\in \mathcal{D}^{\mu}$ is a stopping cube, and $\SSS=\{S\}, S\in \mathcal{D}^{\mu}$
is a collection of its maximal stopping subcubes (we call them stopping suns of $\hat{S}$, there stopping tree distance to $\hat{S}$ is one: $r(S,\hat{S}) =1$), we call $\mathcal{O}_{\hat{S}}$ the collection of all cubes $J$
from  the lattice  $\mathcal{D}^{\nu}$, such that the top side of the cube $Q_J$ built on $J$ as on its base lies in the set 
$\Omega_{\hat{S}} := (\bar{Q}_{\hat{S}}\setminus \cup_{S\in \SSS} \bar{ Q_S})$. In particular, $\hat{S}\in \mathcal{O}_{\hat{S}}$, but its stopping suns are not in $\mathcal{O}_{\hat{S}}$.

\vspace{.2in}

The choice of $\hat{I}$ in a stopping terms above in \eqref{splittingto3terms} is as follows: let $I, J$ be as above, namely $J\subset I, J\in \mathcal{D}_{\nu}$, $J$ good, $J\subset I_i$, where $I_i$ is a son of $I$,
%\tag
we choose the first (and unique) stopping cube $\hat{S}$ such that $I_i\in \mathcal{O}_{\hat{S}}$. Then we just put $\hat{I}=
\hat{S}$.

\vspace{.1in}

\noindent{\bf Definition.} Recall that the father of an cube $I$ with respect to the tree of all dyadic cubes was called $F(I)$. If $S\in\SSS$, then its father with respect to tree $\SSS$ will be always  called from now on $\hat{S}$.

\vspace{.2in}

Let us introduce the sum of absolute values of the ``stopping terms" of the sum $\tau$ above (as always $I\in
\mathcal{D}^{\mu}, J\in \mathcal{D}^{\nu}$).

$$
t:= \sum_{  |J|< 2^{-dr}|I|, J\subset I, J\in\mathcal{D}_{\nu}, J\,\text{is good}}
%\dist (J,e(I)) \geq |I|^{3/4}|J|^{1/4}} 
|\langle \Delta_I^{\mu}f\rangle_{\mu, I_i}|
|(T_{\mu}(\chi_{\hat{I}\setminus I_i}),\Delta_J^{\nu}g)_{\nu}|\,.
$$
To estimate it we can use \eqref{yellow10}.
Then (recall that $I_i$ is the half of $I$ containing $J$)
$$
t\leq A\,\mathcal{T},\,\,\, \mathcal{T}:= \sum_{  |J|< 2^{-dr}|I|, J\subset I,
\, J\,\text{is good}}
\Bigl(\frac{\nu(J)}{\mu(I_i)}\Bigr)^{1/2}\Bigl(\frac{\ell(J)}{\ell(I)}\Bigr)^{\eps/2}\PP_{I_i}(\chi_{\hat{I}\setminus
I_i} \,d\mu)\|\Delta_J^{\nu}g\|_{\nu}\|\Delta_I^{\mu}f\|_{\mu}\,. 
$$

\begin{thm}
\label{shortrangeTWthm}
Let in the sum $\mathcal{T}$ above $\hat{I}$ means the smallest stopping tree cube containing $I_i$. Then
$$
\mathcal{T} \leq c\sqrt{K} \|f\|_{\mu}\|g\|_{\nu}\,.
$$
\end{thm}

\begin{proof}
Put
$$
r_{n,k} := \sum_{|J|< 2^{-dr}|I|, J\subset I_i, \ell(I)= 2^{k}, \ell(J)=
2^{-n+k}}\Bigl(\frac{\nu(J)}{\mu(I_i)}\Bigr)^{1/2}\PP_{I_i}(\chi_{\hat{I}\setminus I_i}
\,d\mu)\|\Delta_J^{\nu}g\|_{\nu}\|\Delta_I^{\mu}f\|_{\mu}\,.
$$
Then abusing slightly the notations we denote the sons of $I$ by $I_1, I_2, \dots$.  We get
$$
r_{n,k} \leq\sum_{i=1}^{2^d} \sum_{\ell(I)= 2^{k}}\|\Delta_I^{\mu}f\|_{\mu}\sum_{J\subset I_i, 
\,\ell(J)=2^{-n+k}}\Bigl(\frac{\nu(J)}{\mu(I_i)}\Bigr)^{1/2}\PP_{I_i}(\chi_{\hat{I}\setminus I_i}
\,d\mu)\|\Delta_J^{\nu}g\|_{\nu}\,.
$$
Consider only $I_1$. By the Cauchy inequality the estimate will be
$$
\sum_{\ell(I)= 2^{k}}\|\Delta_I^{\mu}f\|_{\mu}(\sum_{J\subset I_1, 
\,\ell(J)=2^{-n+k}}\Bigl(\frac{\nu(J)}{\mu(I_1)}\Bigr)[\PP_{I_1}(\chi_{\hat{I}\setminus I_1}
\,d\mu)]^2)^{1/2}(\sum_{J\subset I_1, 
\,\ell(J)=2^{-n+k}}\|\Delta_J^{\nu}g\|_{\nu}^2)^{1/2}
$$
The middle term is bounded by $[\PP_{I_1}(\chi_{\hat{I}\setminus I_1}\,d\mu)]^2\nu(I_1)/\mu(I_1)$. By
\eqref{yellow11}
%%%%%%%%%%%%%%%%%%%%%%%%%%%%%%%%%%%%%%
%, \eqref{Poissonmunu}, and 
%$$
%\mu(I)\leq C_d\mu(I_1)
%$$ 
%(the doubling again), so 
%%%%%%%%%%%%%%%%%%%%%%%%%%%%%%%%%%%%%%%%
we get that the middle term is bounded by $\sqrt{100\,K}$. In fact, this was our choice of $\hat{I}$,
which ensures that $I_1\in \mathcal{O}_{\hat{I}}$, and so \eqref{yellow11} holds.

%%%%%%%%%%%%%%%%%%%%%%%%%%
%%%%%%%%%%%%%%%%%%%%%%%%%%%

Thus, the last expression above is bounded by (this is just the Cauchy inequality)
$$
10\sqrt{K}\sum_{\ell(I)= 2^{k}}\|\Delta_I^{\mu}f\|_{\mu}(\sum_{J\subset I_1, 
\,\ell(J)=2^{-n+k}}\|\Delta_J^{\nu}g\|_{\nu}^2)^{1/2}\leq 
$$
$$
10\sqrt{K}(\sum_{\ell(I)=
 2^{k}}\|\Delta_I^{\mu}f\|_{\mu}^2)^{1/2}
(\sum_{\ell(I)= 2^{k}}\sum_{J\subset I_1, 
\,\ell(J)=2^{-n+k}}\|\Delta_J^{\nu}g\|_{\nu}^2)^{1/2}\,.
$$
As a result we get the estimate on $r_{n,k}$:
$$
r_{n,k} \leq 10\,\sqrt{K} \,(\sum_{\ell(I)=
2^{k}}\|\Delta_I^{\mu}f\|_{\mu}^2)^{1/2}(\sum_{\ell(J)=2^{-n+k}}\|\Delta_J^{\nu}g\|_{\nu}^2)^{1/2}\,.
$$
Now it is obvious from the formulae for $\mathcal{T}$ and $r_{n,k}$  that
$$
\mathcal{T}\leq \sum_n 2^{-n\eps/2}\sum_k r_{n,k}\,.
$$
But from the estimate above and the Cauchy inequality $\sum_k r_{n,k}\leq 10\,\sqrt{K}\,\|f\|_{\mu}\|g\|_{\nu}$.
So we get Theorem \ref{shortrangeTWthm}.

\end{proof}

\section{Difficult terms and several paraproducts}
\label{Paraproducts}

Let us recall $f,g$ are good functions and  that in the sum
\begin{equation}
\label{srsums200}
\tau :=\sum_{  |J|< 2^{-dr}|I|, J\subset I, J\in \mathcal{D}_{\nu}, J\,\text{is good}}
%\dist (J,e(I)) \geq |I|^{3/4}|J|^{1/4}}
(T_{\mu}
\Delta_I^{\mu} f, \Delta_J^{\nu}g)_{\nu} \,.
\end{equation}
we consider each term of $\tau$ and split it to three terms.
To do this, let $I_i$ denote the son of $I$, which contains $J$. And $I_n$ is the union of other sons.
Let $S$ denote the smallest supercube of $I_i$ in the same lattice: $S \in \mathcal{D}^{\mu}$,  $S\in \SSS$ such that
\begin{equation}
\label{defines}
I_i\in \mathcal{O}_{S}\,,
\end{equation}
where the family of cubes $\mathcal{O}_{S}$ was introduced shortly after \eqref{srsums20}. (In other  words $S$ is the smallest stopping cube containing $I_i$.)
%as follows
%\noindent{\bf Notations.} If $S\in \mathcal{D}^{\mu}$ is a stopping cube, and $\SSS_S=\{s\}, s\in
%\mathcal{D}^{\mu}$ is a collection of its maximal stopping subcubes, we call $\mathcal{O}_{S}$ the collection %of
%all cubes $I$ from {\it both} lattices $\mathcal{D}^{\mu}$, $\mathcal{D}^{\nu}$, such that the top side of the %square
%$Q_I$ lies in the set 
%$\Omega_{S} := \bar{ (Q_{S}}\seminus \cup_{s\in \SSS_S} \bar{ Q_s})$
%\tag
\vspace{.2in}

We wrote
$$
(T_{\mu}\Delta_I^{\mu} f, \Delta_J^{\nu}g)_{\nu} = (T_{\mu}(\chi_{I_n}\Delta_I^{\mu} f), \Delta_J^{\nu}g)_{\nu} +
(T_{\mu}(\chi_{I_i}\Delta_I^{\mu} f), \Delta_J^{\nu}g)_{\nu} =
$$
$$
(T_{\mu}(\chi_{I_n}\Delta_I^{\mu} f), \Delta_J^{\nu}g)_{\nu} + \langle \Delta_I^{\mu}f\rangle_{\mu, I_i}
(T_{\mu}(\chi_{S}),\Delta_J^{\nu}g)_{\nu} - \langle \Delta_I^{\mu}f\rangle_{\mu, I_i}
(T_{\mu}(\chi_{S\setminus I_i}),\Delta_J^{\nu}g)_{\nu}\,.
$$
Here $S$ is the smallest cube from the stopping tree $\SSS$ such that $I_i\in \mathcal{O}_{S}$.  Also here $\langle \Delta_I^{\mu}f\rangle_{\mu, I_i}$ is the average of $\Delta_I^{\mu}f$ with respect to $\mu$ over $I_i$,
which is the same as value of this function on $I_i$ (by construction $\Delta_I^{\mu}f$ assumes on $I$ exactly  $2^d$ values, one on each son).

\vspace{.2in}

The sum of absolute values of the first terms and the sum of absolute values of the third terms were already
bounded by $(c_0\sqrt{[\mu,\nu]_{A_2} }+c_1\sqrt{K})\|f||_{\mu}\|g\|_{\nu}$ in the preceding sections. Middle terms were called ``difficult terms", and we are
going to estimate the absolute value of the sum of all difficult terms now. This is the most difficult part of the proof.

\vspace{.2in}

Let $\{S\}_{S\in \SSS}$ denote  the family of stopping cubes of all generations.
{\it In what follows the letter $S$ is reserved for the stopping cubes.}  Recall that $\hat{S}$ also denotes the stopping cube, the father of $S$ inside the stopping tree $\SSS$.

\vspace{.2in}

\noindent{\bf Notations.} Let $S\in\SSS$ be an arbitrary stopping cube. We denote by
$P_{\mu, \mathcal{O}_S}$ the orthogonal projection in $L^2(\mu)$ onto the space generated by
$\{\Delta_I^{\mu}\}$, $I\in \mathcal{O}_S$ (we mean the images of these projector operators),  and we denote by
$\mathbb{P}_{\nu, \mathcal{O}_S}$ the orthogonal projection in $L^2(\nu)$ onto the space generated by
$\{\Delta_J^{\nu}\}$, $J\in \mathcal{O}_S$, $J$ is good (we mean the images of these projector operators). 

\vspace{.2in}

\noindent We fix $I\in \mathcal{D}^{\mu}$, it defines $S\in \SSS$ (see \eqref{defines}),  we look at terms
$$
\langle \Delta_I^{\mu}f\rangle_{\mu, I_i}
(T_{\mu}(\chi_{S}),\Delta_J^{\nu}g)_{\nu}\,.
$$

\noindent We can write each of the term  $\langle \Delta_I^{\mu}f\rangle_{\mu, I_i}
(T_{\mu}(\chi_{S}),\Delta_J^{\nu}g)_{\nu}$ with fixed $S$ and $I\in \mathcal{O}_{S}, J\in
\mathcal{O}_{S}$ as 
$$
\langle \Delta_I^{\mu}P_{\mu, \mathcal{O}_{S}}f\rangle_{\mu,
I_i}(T_{\mu}(\chi_{S}),\Delta_J^{\nu}\mathbb{P}_{\nu,
\mathcal{O}_{S}}g)_{\nu}\,.
$$

\vspace{.2in}

\noindent{\bf The definition of $\tau_{S}$ .}
We collect all of these terms with 
$I\in \mathcal{O}_{S}, I\in \mathcal{D}^{\mu}, J\in \mathcal{O}_{S}$, $J\in \mathcal{D}^{\nu}, |J| \leq 2^{-dr}
|I|$, $J$ is good. The resulting sum is called $\tau_{S}$.
(In summation below we should remember that  $f,g$ are good: so we can sum over all pertinent pairs of $I,J$ remembering
that some of $\Delta$'s are zero anyway.)

\vspace{.2in}

We first fix good $J$, then summing over such $I$'s gives (such $I$'s should contain $J$, and they form a ``tower" of nested
cubes, from the smallest one called $\ell(J)$ to the largest one equal to $S$; notice that the summing of
quantities
$\langle
\Delta_I^{\mu}\f\rangle_{\mu, I}$ over such a ``tower" results in the average over the smallest cube minus the average
over the largest cube of the ``tower", the latter one being zero in our case because the $\mu$-average over $S$ of any $\Delta_L^{\mu}(f)$ with $L\subset S$ is zero; we are dealing only with such $L$'s now, as $L$'s are in $\mathcal{O}_S$ by our definition of $\tau_S$ above).
$$
\langle P_{\mu, \mathcal{O}_{S}}f\rangle_{\mu,
\ell(J)}(\Delta_J^{\nu}T_{\mu}(\chi_{S}),\mathbb{P}_{\nu,
\mathcal{O}_{S}}g)_{\nu}\,,
$$
where $l(J)\in \mathcal{O}_{S},l(J)\in \mathcal{D}^{\mu}, \ell(l(J)) = 2^{r} \ell(J)$.

%%%%%%%%%%
%One can argue that replacing $f$ by $ \mathbb{P}_{\mu, \mathcal{O}_{S}}f$ we make gaps in the tower as $\langle \Delta_I^{\mu}f\rangle_{\mu,I}%$ got replaced by $0$ from time to time (for bad $I$'s actually). But this is not a problem as $f$ is good, and so $\langle \Delta_I^{\mu}f\rangle_
%{\mu,I}$ is zero anyway for bad $I$'s!
%%%%%%%%%%%%
\vspace{.1in}

\subsection{First paraproduct}
\label{FirstParaproduct}

Let us introduce our first paraproduct operator
$$
\pi_{T_{\mu}\chi_{S}} \f := \sum_{I\in\mathcal{D}^{\mu}, I\in  \mathcal{O}_{S}}
\langle \f\rangle_{\mu,I}\sum_{J\in\mathcal{D}^{\nu}, J\in \mathcal{O}_{S}, J\subset I,  \ell(J)=
2^{-r}\ell(I), J\,\text{is good}}\Delta_J^{\nu}T_{\mu}(\chi_{S})\,.
$$
Then the absolute value of the sum $\tau_{S}$ above is
\begin{equation}
\label{C1}
|(\pi_{T_{\mu}\chi_{S}}P_{\mu, \mathcal{O}_{S}}f, \mathbb{P}_{\nu,
\mathcal{O}_{S}}g)_{\nu}|\leq C_1\,\|P_{\mu, \mathcal{O}_{S}}f\|_{\mu} \|\mathbb{P}_{\nu,
\mathcal{O}_{S}}g\|_{\nu}\,,
\end{equation}
where $C_1$ is the norm of $\pi_{T_{\mu}\chi_{S}}$ as an operator from $L^2(\mu)$ to $L^2(\nu)$.

\begin{thm}
\label{FirstParaproduct}
The norm of operator $\pi_{T_{\mu}\chi_{S}}$ as an operator from $L^2(\mu)$ to $L^2(\nu)$ is bounded by
$c\,(\sqrt{K}+\sqrt{K_{\chi}})$.
\end{thm} 

\begin{proof}
Obviously, by orthogonality 
$$
\|\pi_{T_{\mu}\chi_{S}}\f\|_{\nu}^2 \leq \sum_{I\in\mathcal{D}^{\mu}, I\in  \mathcal{O}_{S}}
|\langle \f\rangle_{\mu,I}|^2\,a_I,
$$
where 
$$
a_I := \sum_{J\in\F(I)}\|\Delta_J^{\nu}T_{\mu}(\chi_{S})\|_{\nu}^2\,.
$$
and
$\F(I) := \{ J: J\in\mathcal{D}^{\nu}, J\in \mathcal{O}_{S}, J\subset I,  \ell(J)=
2^{-r}\ell(I), \,J\,\,\text{is good}\}$

The Carleson imbedding theorem (see \cite{G}, and in this context \cite{NTV1}) says that the boundedness of the sum
$\sum_{I\in\mathcal{D}^{\mu}, I\in  \mathcal{O}_{S}}
|\langle \f\rangle_{\mu,I}|^2\,a_I$ by $4C\, \|\f\|_{\mu}^2$ follows from the following Carleson condition
\begin{equation}
\label{CarlC11}
\forall I\in \mathcal{D}^{\mu},\,I\in  \mathcal{O}_{S}\, \sum_{\ell\in\mathcal{D}^{\mu}, \ell\in  \mathcal{O}_{S},
\ell\subset
I}a_{\ell} \leq C\,\mu(I)
\end{equation}
Of course if we put $\Psi(I)  := \{ J: J\in\mathcal{D}^{\nu}, J\in \mathcal{O}_{S}, J\subset I,  |J|\leq
2^{-dr}|I|, J\,\,\text{is good}\}$ we notice that
$$
\sum_{\ell\in\mathcal{D}^{\mu}, \ell\in  \mathcal{O}_{S}, \ell\subset I}a_{\ell} =\sum_{ J:
J\in\Psi(I)}\|\Delta_J^{\nu}T_{\mu}(\chi_{S})\|_{\nu}^2 =\| \sum_{ J:
J\in\Psi(I)}\Delta_J^{\nu}T_{\mu}(\chi_{S})\|_{\nu}^2 \,.
$$
By duality then
$$
\sum_{\ell\in\mathcal{D}^{\mu}, \ell\in  \mathcal{O}_{S}, \ell\subset I}a_{\ell} = \sup_{\psi\in L^2(\nu), \,\|\psi\|_{\nu}
=1}|
\sum_{ J:
J\in\Psi(I)}(T_{\mu}(\chi_{S}), \Delta_J^{\nu}\psi)_{\nu}|^2\leq 
$$
$$
\sup_{\psi\in L^2(\nu), \,\|\psi\|_{\nu} =1}|
\sum_{ J:
J\in\Psi(I)}(T_{\mu}(\chi_{S\setminus I}), \Delta_J^{\nu}\psi)_{\nu}|^2 + \|T_{\mu}(\chi_{ I})\|_{\nu}^2\,.
$$
So \eqref{KKchi} implies
\begin{equation}
\label{CarlC12}
\sum_{\ell\in\mathcal{D}^{\mu}, \ell\in  \mathcal{O}_{S}, \ell\subset I}a_{\ell}  \leq \sup_{\psi\in L^2(\nu),
\,\|\psi\|_{\nu} =1}|
\sum_{ J:
J\in\Psi(I)}(T_{\mu}(\chi_{S\setminus I}), \Delta_J^{\nu}\psi)_{\nu}|^2 + K_{\chi}\,\mu(I)\,.
\end{equation}

Let us consider the term $(T_{\mu}(\chi_{S\setminus I}), \Delta_J^{\nu}\psi)_{\nu}$, $J\in \Psi(I)$. Exactly this
quantity was estimated in  \eqref{yellow9}. We get

$$
|(T_{\mu}(\chi_{S\setminus I}), \Delta_J^{\nu}\psi)_{\nu}|\leq A\, \nu(J)^{1/2}
\|\Delta_J^{\nu}\psi\|_{\nu}\Bigl(\frac{\ell(J)}{\ell(I)}\Bigr)^{\eps/2} \PP_I(\chi_{S\setminus I})\,d\mu\,.
$$
So the first term in \eqref{CarlC12} is bounded by (we use the Cauchy inequality)
$$
\sum_{ J:
J\in\Psi(I)} \bigg(\frac{\ell(J)}{\ell(I)}\bigg)^{\eps}[\PP_I(\chi_{S\setminus I})\,d\mu]^2\nu(J)\leq \sum_n
2^{-n\eps}\sum_{\ell(J)=2^{-n}\ell(I),J\subset I}[\PP_I(\chi_{S\setminus I})\,d\mu]^2\nu(J) =
$$
$$
\sum_n
2^{-n\eps}[\PP_I(\chi_{S\setminus I})\,d\mu]^2\nu(I)
$$
as $\|\psi\|_{\nu} =1$. It is time to use the fact that $I\in \mathcal{O}_{S}$, 
which means that the stopping criterion
\eqref{yellow11} is not yet achieved on $I$, in other words that
$$
[\PP_I(\chi_{S\setminus I})\,d\mu]^2\nu(I)\leq 100\, K\, \mu(I)\,.
$$
Combining this with \eqref{CarlC12} we get \eqref{CarlC11}:
$$
\sum_{\ell\in\mathcal{D}^{\mu}, \ell\in  \mathcal{O}_{S}, \ell\subset I}a_{\ell}  \leq c\,(K+K_{\chi})\,\mu(I)\,.
$$
And Theorem \ref{FirstParaproduct} is proved as the norm (as we already remarked) of our paraproduct operator is the square root of $4$ times the constant in the previous inequality.

\end{proof}

Let us recall that we introduced above the definition of $\tau_{S}$, for stopping cube $S$.
We finished the estimate of 
the sum of $\tau_{S}$ over all stopping $S$ (recall that the set of all, stopping cubes was called $\SSS$):

\begin{equation}
\label{inSinS}
\sum_{S\in \SSS}\tau_{S} \leq c\,(\sqrt{K} +\sqrt{K_{\chi}}) \sum_{S\in\SSS}\|P_{\mu,
\mathcal{O}_{S}}f\|_{\mu}\|\mathbb{P}_{\nu,
\mathcal{O}_{S}}g\|_{\nu}\leq  c\,(\sqrt{K} +\sqrt{K_{\chi}})\|f\|_{\mu}\|g\|_{\nu},,
\end{equation}
the last inequality following from the orthogonality of $P_{\mu,
\mathcal{O}_{S}}f$ for different $S$ (the same for $\mathbb{P}_{\nu,
\mathcal{O}_{S}}g$) and the Cauchy inequality.

\subsection{Careful bookkeeping: two more paraproducts }
\label{Twomoreparaproducts}

\noindent{\bf Definition.}  Similarly to $\mathbb{P}_{\nu, \mathcal{O}_S}$ defined above we define $\mathbb{P}_{\nu, Q_S}$ and $\mathbb{P}_{\nu, Q_S\setminus\mathcal{O}_S}$ as projections on the sum of $\Delta_J^{\nu}$ with good $J$ such that $J$ lies in $Q_S$ and $Q_s\setminus \mathcal{O}_S$ correspondingly. 

In the previous subsection we have estimated a piece of the sum of the difficult terms
\begin{equation}
\label{rest}
\langle \Delta_I^{\mu}f\rangle_{\mu, I_i}
(T_{\mu}(\chi_{S}),\Delta_J^{\nu}g)_{\nu}\,,
\end{equation}
namely, we estimated the sum of such terms, when $I,J$ lie both in the same family
 $\mathcal{O}_{S}$, where $S\in \SSS$ (arbitrary stopping cube). Such a sum was called $\tau_S$, and we just proved in
\eqref{inSinS} that
$\sum_{S\in\SSS} \tau_S \leq c(\sqrt{K}+\sqrt{K_{\chi}})\|f\|_{\mu}\|g\|_{\nu}$.

What is left is to estimate the sum of abovementioned terms when $J\in \mathcal{O}_{S}$ and $I$ belongs to another
$\mathcal{O}_{S_1}$, where $S, S_1$ are both stopping cubes. As $I$ is larger than $J$, we have to consider the
pairs of stopping cubes, where $S$ is strictly inside $S_1$ ($S_1$ is one or more  generations higher in a stopping tree $\SSS$ than $S$).

Let us recall that $F(I)$ denote the father of $I$ inside the standard dyadic tree.
Let us fix $J$. Let $S_j\subset S_{j-1} \subset S_1=I_0$ be the whole (finite) sequence of stopping cubes
of successive generations containing $J$.  So $S_{i-1}$ is a father of $S_i$ in the stopping tree $\SSS$. Hence, it is not true that $S_{i-1}=F(S_i)$ in general!
%%%%%%%%%%%%%%
%The sequence for $I$'s, over which we have to sum up,  will be one term shorter
%(the smallest one should be discarded). This is  because we sum up all the terms, where $J$ and $I$ are in different families
%$\mathcal{O}_{S_i}, \mathcal{O}_{S_{i-1}}$, and $S_i$ is inside $S_{i-1}$.
%%%%%%%%%%%%%%%%
 Notice also that $\langle \Delta_I^{\mu}f\rangle_{\mu, I_i}$
is the difference between two averages of $f$ with respect to $\mu$, one over $I_i$ and one over its father $I$. It is easy 
to some up successive differences and summing all above mentioned terms with fixed $J$ we get (omitting for brevity the common factor of the scalar product: $\Delta_J^{\nu}(g)$):
$$
(\langle f\rangle_{\mu, F(S_j)} -\langle f\rangle_{\mu, F^2(S_j)} ) T_{\mu}\chi_{S_{j-1}}+\dots + (\langle f\rangle_{\mu, S_{j-1}} -\langle f\rangle_{\mu, F(S_{j-1})} ) T_{\mu}\chi_{S_{j-1}} +
$$
$$
(\langle f\rangle_{\mu, F(S_{j-1})} -\langle f\rangle_{\mu, F^2(S_{j-1})} ) T_{\mu}\chi_{S_{j-2}}+\dots + (\langle f\rangle_{\mu, S_{j-2})} -\langle f\rangle_{\mu, F(S_{j-2})} ) T_{\mu}\chi_{S_{j-2}} +
$$
$$
\cdots +
$$
$$
(\langle f\rangle_{\mu, F(S_2)} -\langle f\rangle_{\mu, F^2(S_2)} ) T_{\mu}\chi_{S_{1}}+\dots + (\langle f\rangle_{\mu, S_{1}} -\langle f\rangle_{\mu, F(S_{1})} ) T_{\mu}\chi_{S_{1}} \,.
$$

Recall that we are working with $f$'s such that two last averages will be zero. Regrouping we obtain

$$
\langle f\rangle_{\mu, F(S_{j-1})} T_\mu\chi_{S_{j-2}\setminus S_{j-1}} +\cdots + \langle f\rangle_{\mu, F(S_{2})} T_\mu\chi_{S_{1}\setminus S_{2}} 
$$
and
$$
\sum_k \langle f\rangle_{\mu, F(S_{k})} T_\mu\chi_{S_{k-1}}= \sum_k \langle f\rangle_{\mu, F(S_{k})} T_\mu\chi_{\hat{S_{k}}}\,.
$$

Let us consider the first sum and let us now  collect all pertinent $J$'s. Because in the first sum $J\in \mathcal{O}_{S_j}$ and averages are over $S$ with indices strictly smaller than $j$ we obtain the following sum by collecting:

$$
\pi^{(1)}_{\SSS} (f,g):= \sum_{S\in \SSS} \langle f \rangle_{F(S)}(T_{\mu}\chi_{\hat{S}\setminus S}, \mathbb{P}_{\nu, Q_S\setminus\mathcal{O}_S}g)\,.
$$
Let us consider the second sum and let us now  collect all pertinent $J$'s. We get
$$
\pi^{(2)}_{\SSS} (f,g):= \sum_{S\in \SSS} \langle f \rangle_{F(S)}(T_{\mu}\chi_{\hat{S}}, \mathbb{P}_{\nu, \mathcal{O}_S}g)\,.
$$

However, there is also $\pi^{(3)}_{\SSS} (f,g)$ because so far we collected all difficult terms such that
$$
J\in \mathcal{O}_S\,,\,\, I_i\in \mathcal{O}_{S'}\,,\,\, S\subset S'\,,\, S\neq S'\,,\,\, r(S, S')\ge 1\,.
$$
But we have to collect also the difficult terms such that $J$ and $I_i$ are in the same $\mathcal{O}_S$  but $I$ is already not in it:
$$
J, I_i\in \mathcal{O}_S\,,\,\, I\in \mathcal{O}_{\hat{S}}\,,\,\, \hat{S} \,\,\text{is the stopping father of}\,\, S\,,\,\,\text{that is the terms with}\,\, I_i =S\,.
$$
This gives us terms (in the previous notations $I_i=S_j$, $I=F(S_j)$, $\hat{I}=S_j$, the last equality is just exactly how we chose $\hat{I}$ in the definition of difficult terms, these are situations when we do not have stopping terms, they vanish)
$$
(\langle f\rangle_{\mu, I_i}-\langle f\rangle_{\mu, I}) T_{\mu}\chi_{S_j}=(\langle f\rangle_{\mu,S_j}-\langle f\rangle_{\mu, F(S_j)}) T_{\mu}\chi_{S_j}\,.
$$
Collecting we obtain
$$
\pi^{(3)}_{\SSS} (f,g):=  \sum_{S\in \SSS} \langle f \rangle_{S}(T_{\mu}\chi_{S}, \mathbb{P}_{\nu, \mathcal{O}_S}g)-\sum_{S\in \SSS} \langle f \rangle_{F(S)}(T_{\mu}\chi_{S}, \mathbb{P}_{\nu, \mathcal{O}_S}g)\,.
$$

Now we can consider sum of all difficult terms $\rho=\pi^{(1)}_{\SSS}(f,g)+\pi^{(2)}_{\SSS}(f,g)+\pi^{(3)}_{\SSS}(f,g)$ by uniting the sum $-\sum_{S\in \SSS} \langle f \rangle_{F(S)}(T_{\mu}\chi_{S}, \mathbb{P}_{\nu, \mathcal{O}_S}g)$ with $\pi^{(2)}_{\SSS}(f,g)$, and then uniting the result with  $\pi^{(1)}_{\SSS}(f,g)$. The sum $\sum_{S\in \SSS} \langle f \rangle_{S}(T_{\mu}\chi_{S}, \mathbb{P}_{\nu, \mathcal{O}_S}g)$ stays alone:
$$
\rho = \sum_{s\in\SSS} \langle f\rangle_{\mu, F(S)} (T_{\mu} \chi_{\hat{S} \setminus S},\mathbb{P}_{\nu, Q_S}g )_{\nu} +
\sum_{s\in\SSS} \langle f\rangle_{\mu, S} (T_{\mu} \chi_{S}, \mathbb{P}_{\nu, \mathcal{O}_S}g)_{\nu} =: \rho_1 +\rho_2\,.
$$

\vspace{.2in}

We introduce now  two paraproducts:

$$
\pi^{\mathcal{O}} f: = \sum_{s\in\SSS} \langle f\rangle_{\mu, S} \mathbb{P}_{\nu, \mathcal{O}_S}(T_{\mu} \chi_{S})\,,
$$
$$
\pi^{Q}f : = \sum_{s\in\SSS} \langle f\rangle_{\mu, F(S)} \mathbb{P}_{\nu, Q_S}(T_{\mu} \chi_{\hat{S}\setminus S})\,.
$$
Then $\rho_1 = (\pi^{\mathcal{O}} , g)_{\nu}, \rho_2 = (\pi^{Q} , g)_{\nu}$. So to finish the proof of our Theorem
\ref{pivotal3} it is enough to prove the boundedness of these paraproducts as operators from $L^2(\mu)$ to
$L^2(\nu)$ with the estimate of norm $\leq c\,(\sqrt{K}+\sqrt{K_{\chi}})$.

To prove the boundedness of the first paraproduct let us use Theorem \ref{yellowS}. Consider the sequence
$$
\{b_S\}_{S\in\SSS},\,\, b_S := \|\mathbb{P}_{\nu, \mathcal{O}_S}(T_{\mu} \chi_{S})\|_{\nu}^2 \,.
$$
It is a Carleson sequence:
\begin{equation}
\label{CarlaS}
\forall I\in \mathcal{D}^{\mu}\,\,\sum_{S\subset I, S\in \SSS} b_S \leq c\,K_{\chi}\, \mu(I)\,.
\end{equation}
In fact, $ b_S \leq  \|T_{\mu} \chi_{S}\|_{\nu}^2\leq K_{\chi}\, \mu(S)$ by \eqref{Kchi}. Now \eqref{CarlaS} becomes clear
by Theorem \ref{yellowS}.

Notice that $\mathbb{P}_{\nu, \mathcal{O}_S}$ are mutually orthogonal projections in $L^2(\nu)$ for different $S$. 
This is just because the families $\mathcal{O}_S$ are pairwise disjoint for different $S\in \SSS$.
This is exactly what helped us to cope with $\pi^{\mathcal{O}} f$ so easily, we just used
$$
\|\pi^{\mathcal{O}} f\|_{\nu}^2= \|\sum_{s\in\SSS} \langle f\rangle_{\mu, S} \mathbb{P}_{\nu, \mathcal{O}_S}(T_{\mu} \chi_{S})\|_{\nu}^2= \sum_{S\in\SSS}| \langle f\rangle_{\mu, S}|^2\| \mathbb{P}_{\nu, \mathcal{O}_S}(T_{\mu} \chi_{S})\|_{\nu}^2=  \sum_{S\in\SSS}| \langle f\rangle_{\mu, S}|^2 a_S\,.
$$
This is where the orthogonality has been used. And we applied then the Carleson property of $\{b_S\}_{S\in\SSS}$.
We
already saw this type of paraproducts with the property of orthogonality (see
\cite{G},
\cite{NTV1}, and especially Theorem
\ref{FirstParaproduct} above). And we know  that Carleson condition
\eqref{CarlaS} is sufficient for the paraproduct operator $\pi^{\mathcal{O}}$ to be bounded with constant $2\sqrt{c\,K_{\chi}}$.

\vspace{.1in}

\noindent The second paraproduct $\pi^Q$ is a quite different story because projections $\mathbb{P}_{\nu Q_S}\,,\, S\in \SSS$ are not orthogonal.

\vspace{.2in}

\subsection{ The second paraproduct $\pi^{Q}$.}
\label{secondpp}

\vspace{.2in}

So $\|\pi^{Q} f\|_{\nu}^2$ has the diagonal part but also the out of diagonal par:
$$
\|\pi^{Q} f\|_{\nu}^2 \leq  DP + ODP\,,
$$
where
$$
DP := \sum_{S\in \SSS} |\langle f\rangle_{\mu,F(S)}|^2 \|\mathbb{P}_{\nu, Q_S} T_{\mu}(\chi_{\hat{S}\setminus S})\|_{\nu}^2\,,
$$
$$
ODP := \sum_{S, S'\in \SSS, S'\subset S, S'\neq S}|\langle f\rangle_{\mu,F(S')}||\langle f\rangle_{\mu,F(S)}| |(\mathbb{P}_{\nu,
Q_S} T_{\mu}(\chi_{\hat{S}\setminus S}), \mathbb{P}_{\nu, Q_{S'}} T_{\mu}(\chi_{\hat{S'}\setminus S'})_{\nu}| =
$$
$$
\sum_{S, S'\in \SSS, S'\subset S, S'\neq S}|\langle f\rangle_{\mu,F(S')}||\langle f\rangle_{\mu,F(S)}| |(\mathbb{P}_{\nu,
Q_{S'}} T_{\mu}(\chi_{\hat{S}\setminus S}), \mathbb{P}_{\nu, Q_{S'}} T_{\mu}(\chi_{\hat{S'}\setminus S'})_{\nu}|\,.
$$
We start with $ODP$. Recall that $r=r(S', S)$ is the generation gap between $S'$ and $S$, $S'\subset S$ in the stopping tree $\SSS$. Choose a small $\e_0$ depending on $\eps$ of \cz assumptions:
$$
ODP \leq \sum_{S, S'\in \SSS, S'\subset S, S'\neq S}|\langle f\rangle_{\mu,F(S)}|^2 \|\mathbb{P}_{\nu, Q_{S'}}
T_{\mu}(\chi_{\hat{S}\setminus S})\|_{\nu}^2\cdot (1+\e_0)^{r(S',S)} + 
$$
$$
\sum_{S, S'\in \SSS, S'\subset S, S'\neq S}
|\langle f\rangle_{\mu,F(S')}|^2 \|\mathbb{P}_{\nu, Q_{S'}}
T_{\mu}(\chi_{\hat{S'}\setminus S'})\|_{\nu}^2\cdot (1+\e_0)^{-r(S',S)}\leq
$$
$$
\sum_{S\in\SSS} |\langle f\rangle_{\mu,F(S)}|^2\sum_{j=1}^{\infty} (1+\e_0)^j \sum_{S'\in \SSS, S'\subset S, r(S',S)
=j}\|\mathbb{P}_{\nu, Q_{S'}} T_{\mu}(\chi_{\hat{S}\setminus S})\|_{\nu}^2 + C(\e_0) \sum_{S\in\SSS} |\langle
f\rangle_{\mu,F(S)}|^2 \|\mathbb{P}_{\nu, Q_{S}} T_{\mu}(\chi_{\hat{S}\setminus S})\|_{\nu}^2\,.
$$

Now we need to estimate these  sums
$$
F_j := \sum_{S\in\SSS} |\langle f\rangle_{\mu,F(S)}|^2\sum_{S'\in \SSS, S'\subset S, r(S',S)
=j}\|\mathbb{P}_{\nu, Q_{S'}} T_{\mu}(\chi_{\hat{S}\setminus S})\|_{\nu}^2, \,\, j=1,2,3,...\,,
$$
$$
F_0 := \sum_{S\in\SSS} |\langle
f\rangle_{\mu,F(S)}|^2 \|\mathbb{P}_{\nu, Q_{S}} T_{\mu}(\chi_{\hat{S}\setminus S})\|_{\nu}^2\,.
$$
By the way, $F_0=DP$.  We need to see that $F_j$ are exponentially small.

All such sums have the form of Carleson imbedding theorems. So we need to check countable number of Carleson 
conditions now.

\vspace{.2in}

\noindent{\bf Carleson condition for $F_j$}.
We introduce the sequence
$$
a_S :=  \|\mathbb{P}_{\nu, Q_{S}} T_{\mu}(\chi_{\hat{S}\setminus S})\|_{\nu}^2,\,\,S, \hat{S}\in\SSS, r(S, \hat{S})=1\,.
$$
And also
$$
a^j_S :=\sum_{S'\in\SSS, S'\subset S, r(S', S) =j} \|\mathbb{P}_{\nu, Q_{S'}} T_{\mu}(\chi_{\hat{S}\setminus S})\|_{\nu}^2,\,r(S, \hat{S})=1,\,\,\, j=1,2,3,...\,.
$$
%%%%%%%%%%%%%%%%%%%%%%%%%%%%%%%%%%%%
%\tag
We will need the following Lemma.

\begin{lm}
\label{projectionviaPoisson}
Let $S'\subseteq S\subset \hat{S}$ be cubes of $\mathcal{D}_{\mu}$. Let the tree distance between $S'$ and $S$ with respect to the tree $\mathcal{D}_{\mu}$ satisfy $t(S',S) \geq j,\, j= 0, 1, 2,..$. Then
$$
\|\bP_{\nu, S'} (T_{\mu}\chi_{\hat{S}\setminus S})\|_{\nu}^2 \leq C\, 2^{-j\eps/2} \nu(S')(\PP_S\chi_{\hat{S}\setminus S} d\mu)^2\,.
$$
\end{lm}

\begin{proof}
Let $\|\psi\|_{\nu}=1$. Let us consider the term $(T_{\mu}(\chi_{\hat{S}\setminus S }), \Delta_J^{\nu}\psi)_{\nu}$, $J\in
Q_{S'}$. Exactly this quantity was estimated in  \eqref{yellow9}. We get

$$
|(T_{\mu}(\chi_{\hat{S}\setminus S}), \Delta_J^{\nu}\psi)_{\nu}|\leq c\, \nu(J)^{1/2}
\|\Delta_J^{\nu}\psi\|_{\nu}\Bigl(\frac{\ell(J)}{\ell(S)}\Bigr)^{\eps/2}\PP_S(\chi_{\hat{S}\setminus S}\,d\mu)\,.
$$
So each our projection can be estimated as follows
\begin{equation}
\label{starstar}
\|\mathbb{P}_{\nu, Q_{S}} T_{\mu}(\chi_{\hat{S}\setminus S})\|_{\nu}^2\leq (\PP_S(\chi_{\hat{S}\setminus S})\,d\mu)^2 \sum_{J\, \text{good}, 
J\subset S'}
\nu(J)\bigg(\frac{\ell(J)}{\ell(S)}\bigg)^{\eps/2} \,.
\end{equation}
So $\|\bP_{\nu, S'} (T_{\mu}\chi_{\hat{S}\setminus S}\|_{\nu}^2 $ is bounded by
$$
 (\PP_S(\chi_{\hat{S}\setminus S})\,d\mu)^2 \sum_{t=j}^{\infty}\sum_{\ell(J)=2^{-t}\ell(S), J\subset
S'}
\nu(J)\bigg(\frac{\ell(J)}{\ell(S)}\bigg)^{\eps/2} \,.
$$
%$$
%\sum_{S\in\SSS, S\subset \hat{S}, r(S,\hat{S}) =1} (P_S(\chi_{\hat{S}\setminus S})\,d\mu)^2 \nu(S) \leq K\, %\sum_{S\in\SSS,
%S\subset \hat{S}, r(S,\hat{S}) =1}\mu(S) \leq  K\,\mu(\hat{S})\,,
%$$
which proves the lemma.
\end{proof}

%%%%%%%%%%%%%%%%%%%%%%%%%%%%%%%%%%
We first establish a Carleson property for $\{a_S\}$.  
%By Theorem \ref{yellowS} it is enough to fix a stopping
%$\hat{S}$ and to prove
%\tag
Let $I$ be in $\mathcal{D}_{\mu}$. 

We need to prove
\begin{equation}
\label{againstarallI}
\sum_{S\in\SSS, F(S)\subset  I} \|\mathbb{P}_{\nu, Q_{S}} T_{\mu}(\chi_{\hat{S}\setminus S})\|_{\nu}^2
\leq C\,K \mu(I)\,.
\end{equation}
Consider the family of our largest stopping cubes $\{S_{\al}\}_{\al \in A}$ such that $F(S_{\al})\subset I$. 
We consider their father. We call it $\hat{S}$ abusing the notations slightly.

Using our notations for father in the stopping tree $\SSS$ we can write
$$
\hat{S_{\al}} = \hat{S}\,\,\,\forall \al \in A\,.
$$
%%%%%%%%%%%%%%%%
%There can be  a case that such family consists of one cube (call it $S_0$) and $S_0=I$. Consider this case later. Now we assume that all $S_%{\al}, al\in A$ are strictly smaller than $I$, and therefore
%$$
%F(S_{\al})\subset\hat{S}\,\,\,\forall \al \in A\,.
%$$
%%%%%%%%%%%%%%%%%%
Notice that
\begin{equation}
\label{goodwithF}
 (\PP_{S_{\al}}(\chi_{\hat{S}\setminus F(S_{\al})})\,d\mu)^2\nu(F(S_{\al}))\leq 100\,K\mu(F(S_{\al}))\,\,\,\forall \al\in A\,.
 \end{equation}
 But this is not true with replacing $F(S_{\al})$ by $S_{\al}$!
 Let us use naively \eqref{goodwithF} and Lemma \ref{projectionviaPoisson}. Then we get
 $$
  \sum_{\al\in A} \|\bP_{\nu, Q_{S_{\al}}} (T_{\mu}\chi_{\hat{S}\setminus S_{\al}}\|_{\nu}^2 \leq 2  \sum_{\al\in A} \|\bP_{\nu, Q_{S_{\al}}} (T_{\mu}\chi_{\hat{S}\setminus F(S_{\al})}\|_{\nu}^2 + 2 \sum_{\al\in A}  \|\bP_{\nu, Q_{S_{\al}}} (T_{\mu}\chi_{F(S_{\al})\setminus S_{\al}}\|_{\nu}^2 \leq
$$
$$
 200\,K \sum_{\al\in A}\mu(F(S_{\al})) +2 \sum_{\al\in A}\|(T_{\mu}\chi_{F(S_{\al})}\|_{\nu}^2 \leq (200\,K +K_{\chi}) \sum_{\al\in A}\mu(F(S_{\al})) \,.
 $$
 In other words we would like to conclude that
\begin{equation}
\label{againstar}
\sum_{\al\in A} \|\mathbb{P}_{\nu, Q_{S_{\al}}} T_{\mu}(\chi_{\hat{S}\setminus S_{\al}})\|_{\nu}^2
\leq c\,(K+K_{\chi})\, \mu(I)\,.
\end{equation}
 But instead, by naive reasoning we achieved
 \begin{equation}
\label{falseagainstar}
\sum_{\al\in A} \|\mathbb{P}_{\nu, Q_{S_{\al}}} T_{\mu}(\chi_{\hat{S}\setminus S_{\al}})\|_{\nu}^2
\leq c\,(K+K_{\chi})\,\sum_{\al\in A}\mu(F(S_{\al}))\,.
\end{equation}
 This is a dangerous place because while the cubes $S_{\al}$ are pairwise disjoint, there fathers
 $F(S_{\al})$'s are usually not and we cannot deduce \eqref{againstar} from \eqref{falseagainstar}, as this is not guaranteed that
 $$
 \sum_{\al\in A}\mu(F(S_{\al})) \leq c\,\mu(I)\,.
 $$
 We cannot use doubling. This last inequality actually is usually false.
 
 However, \eqref{againstar} is true. But the way to prove it is more subtle. Let us do it.
 Let $\{F_{\beta}\}_{\beta\in B}$ denote the family of {\it maximal} cubes among $\{F(S_{\al})\}_{\al\in A}$.
 Let for a given $\beta \in B$ the family $\{S_{\beta,\gamma}\}$ denote all cubes from $\{S_{\al}\}_{\al\in A}$ that lie in $F_{\beta}$. Now
 $$
\sum_{\al\in A} \|\bP_{\nu, Q_{S_{\al}}} (T_{\mu}\chi_{\hat{S}\setminus S_{\al}})\|_{\nu}^2 =
 \sum_{\beta\in B} \sum_{\gamma}\|\bP_{\nu, Q_{S_{\beta,\gamma}}} (T_{\mu}\chi_{\hat{S}\setminus S_{\beta,\gamma}})\|^2_{\nu}\leq
$$
$$
2\sum_{\beta\in B} \sum_{\gamma}\|\bP_{\nu, Q_{S_{\beta,\gamma}}} (T_{\mu}\chi_{\hat{S}\setminus F_{\beta}})\|^2_{\nu} + 2 \sum_{\beta\in B} \sum_{\gamma}\|\bP_{\nu, Q_{S_{\beta,\gamma}}} (T_{\mu}\chi_{F_{\beta}\setminus S_{\beta,\gamma} })\|^2_{\nu}=: \Sigma_1 +\Sigma_2\,.
$$
%%%%%%%%%%%%%%%%%%%%%%
For the second sum:
$$
\sum_{\gamma}\|\bP_{\nu, Q_{S_{\beta,\gamma}}} (T_{\mu}\chi_{F_{\beta}\setminus S_{\beta,\gamma} })\|^2_{\nu}\leq 2\sum_{\gamma}\|\bP_{\nu, Q_{S_{\beta,\gamma}}} (T_{\mu}\chi_{F_{\beta}})\|_{\nu}^2 +
$$
$$
 2\sum_{\gamma} \| T_{\mu}\chi_{S_{\beta,\gamma} }\|^2_{\nu}\leq 4\,K_{\chi}\mu(F_{\beta})
 $$
by our Sawyer's type test assumption \eqref{Kchi}. Also we can use now the disjointness of $F_{\beta}$ to conclude that 
$$
\Sigma_2 \leq 4\,K_{\chi}\,\mu(I)\,.
$$
For the first sum we use Lemma \ref{projectionviaPoisson} to conclude
$$
\Sigma_1 \leq \sum_{\beta\in B} \sum_{\gamma} (\PP_{F_{\beta}}\chi_{\hat{S}\setminus F_{\beta}}d\mu)^2\nu(S_{\beta,\gamma} )\leq 
$$
$$
 \sum_{\beta\in B}  (\PP_{F_{\beta}}\chi_{\hat{S}\setminus F_{\beta}}d\mu)^2 \nu(F_{\beta}) \leq K \sum_{\beta\in B} \mu(F_{\beta})\leq 100\,K\,\mu(I)\,.
$$

We used  here the disjointness twice.

For the second sum we can use another estimate--via $K$, not $K_{\chi}$--if we use  Lemma \ref{projectionviaPoisson} again: 
$$
\Sigma_2=\sum_{\beta\in B}\sum_{\gamma}\|\bP_{\nu, Q_{S_{\beta,\gamma}}} (T_{\mu}\chi_{F_{\beta}\setminus S_{\beta,\gamma} })\|^2_{\nu}\leq 
$$
$$
\sum_{\beta\in B}\sum_{\gamma} (\PP_{S_{\beta,\gamma}}\chi_{F_{\beta}\setminus S_{\beta,\gamma} })d\mu)^2 \nu(S_{\beta,\gamma}) \leq K \sum_{\beta\in B} \mu(F_{\beta})\le K\,\mu(I)\,.
$$
We used here \eqref{PIVOTAL}.  (We use it here for the second time in our proof, the first one was in Theorem \ref{yellowS}, notice that the first method of estimate $\Sigma_2$ does not require the use of \eqref{PIVOTAL}, but instead involves constant $K_{\chi}$, here we use only constant $K$, it may be important for something.)

Finally \eqref{againstar} is proved. But to prove the estimate of Carleson type for $\{a_S\}_{S\in\SSS}$ we need
not just \eqref{againstar} but
\begin{equation}
\label{againstarall}
\sum_{S\in\SSS, F(S)\subset  I} \|\mathbb{P}_{\nu, Q_{S}} T_{\mu}(\chi_{\hat{S}\setminus S})\|_{\nu}^2
\leq C\,K \mu(I)\,.
\end{equation}
We estimated not the whole sum in \eqref{againstarall} but only the sum over {\it maximal} $S$ such that $S\in\SSS, F(S)\subset  I$. 
In other words we estimated
\begin{equation}
\label{againstarmax}
\sum_{S_{\al}\in\SSS, F(S_{\al})\subset  I, S_{\al} \,\text{is maximal}} \|\mathbb{P}_{\nu, Q_{S_{\al}}} T_{\mu}(\chi_{\hat{S_{\al}}\setminus S_{\al}})\|_{\nu}^2
\leq C\, K\mu(I)\,.
\end{equation}
%%%%%%%%%%%%%%%%
But the standard reasoning shows that \eqref{againstarmax} is enough to prove \eqref{againstarall}!
In fact, if our $S$ in the sum in \eqref{againstarall} is not maximal it is contained in a maximal one.
Denoting by $S_j(\al)$ the maximal such $S$ contained in $S_{\al}$ we conclude
$$
\sum_j  \|\mathbb{P}_{\nu, Q_{S_j(\al)}} T_{\mu}(\chi_{\widehat{S_j(\al)}\setminus S_j(\al)})\|_{\nu}^2
\leq C\,K\mu(S_{\al})\,.
$$
We sum over $j$ and $\al$ and notice that our main stopping property says
$$
\sum_{\al} \mu(S_{\al}) \leq  \mu(I)\,.
$$
This gives the sum over maximal cubes inside maximal cubes.
Next generation of stopping cubes will give a contribution $ \frac12 \mu(I)$ because 
$$
\sum_{\al}\sum_j\mu(S_j(\al))\leq \frac12\sum_{\al} \mu(S_{\al})\leq \frac12 \mu(I)\,,
$$
yet next generation will come with the contribution  $ \frac14 \mu(I)$ et cetera...
All this is because of Theorem \ref{yellowS}.
And we obtain  \eqref{againstarall}.

\noindent This gives
\begin{equation}
\label{F0est}
DP=F_0 \leq C\,K\,\|f\|_{\mu}^2\,.
\end{equation}
\noindent We are left to estimate $ODP$ or rather to give an exponentially decaying estimates of sums $F_j$.

\subsection{Miraculous improvement of the Carleson property of the sequence $\{a^j_S\}_{S\in\SSS}$}
\label{Miracle}

We used Lemma \ref{projectionviaPoisson} above. But we used it only with $j=0$.
Now we will be estimating Carleson constant for $\{a_S^j\}_{S\in\SSS}$ and it should be exponentially small.
We will use again Lemma \ref{projectionviaPoisson}  but with $j>0$. Recall that $r(S',S)$ denote the tree distance between these two cubes  {\it inside the stopping tree}.  We again consider $I\in \mathcal{D}_{\mu}$, the smallest $\hat{S}\in \SSS$ containing $I$. We need now the estimate
\begin{equation}
\label{againstarall}
\sum_{S\in\SSS, F(S)\subset  I}\sum_{S'\subset S, r(S',S) =j} \|\mathbb{P}_{\nu, Q_{S'}} T_{\mu}(\chi_{\hat{S}\setminus S})\|_{\nu}^2
\leq C\,2^{-cj}\, \mu(I)\,.
\end{equation}

We repeat verbatim the reasoning of the previous section, and of course $2^{-j\eps/2}$ appears naturally
from Lemma \ref{projectionviaPoisson}. We just use the fact that cubes $S'$ involved in $\mathbb{P}_{\nu, Q_{S'}} $  have the property
$$
t(S',S)\geq  r(S', S) \geq j\,.
$$

The only place where one should be careful to get the extra $2^{-cj}$  is the estimate of $\Sigma_2$. Now the estimate via $K_{\chi}$ will not work.
We cannot use 
$$
\sum_{\gamma}\sum_{S'\subset S_{\beta,\gamma}, r(S',S_{\beta,\gamma})=j}\|\bP_{\nu, S'} (T_{\mu}\chi_{F_{\beta}\setminus S_{\beta,\gamma} })\|^2_{\nu}\leq 
$$
$$
2\sum_{\gamma} \sum_{S'\subset S_{\beta,\gamma}, r(S',S_{\beta,\gamma})=j}\|\bP_{\nu, S'} (T_{\mu}\chi_{F_{\beta}})\|_{\nu}^2
+2\sum_{\gamma}\|
T_{\mu}\chi_{S_{\beta,\gamma} }\|^2_{\nu}\leq K_{\chi}\mu(F_{\beta})
$$ 
anymore. Actually we can use it but this does not give extra $2^{-cj}$.  Instead,  we use our second method of estimating $\Sigma_2$ by Lemma \ref{projectionviaPoisson}:
$$
\sum_{\gamma}\sum_{S'\subset S_{\beta,\gamma}, r(S',S_{\beta,\gamma})=j}\|\bP_{\nu, S'} (T_{\mu}\chi_{F_{\beta}\setminus S_{\beta,\gamma} })\|^2_{\nu}\leq 
$$
$$
C\,2^{-j\eps/2}\sum_{\gamma} \sum_{S'\subset S_{\beta,\gamma}, r(S',S_{\beta,\gamma})=j}\nu(S')
(P_{S_{\beta, \gamma}}\chi_{F_{\beta} \setminus S_{\beta,\gamma}}d\mu)^2\leq
$$
$$
C\,2^{-j\eps/2}\sum_{\gamma}\nu(S_{\beta,\gamma})(P_{S_{\beta, \gamma}}\chi_{F_{\beta} \setminus S_{\beta,\gamma}}d\mu)^2
$$
For a fixed $\beta$, the cubes $S_{\beta,\gamma}$ are disjoint by their construction (see above).
It is time to use \eqref{PIVOTAL}.  (We use it here for the third time in our proof, the first one was in Theorem \ref{yellowS}, the second time was the estimate of $\Sigma_2$ for $DP$ above via $K$.)  If we apply \eqref{PIVOTAL} to the last sum, we get
$$
\sum_{\gamma}\nu(S_{\beta,\gamma})(P_{S_{\beta, \gamma}}\chi_{F_{\beta} \setminus S_{\beta,\gamma}}d\mu)^2\leq
K\, \mu(F_{\beta})\,.
$$
Therefore,
 $$
 \sum_{\gamma}\sum_{S'\subset S_{\beta,\gamma}, r(S',S_{\beta,\gamma})=j}\|\bP_{\nu, S'} (T_{\mu}\chi_{F_{\beta}\setminus S_{\beta,\gamma} })\|^2_{\nu}\leq
c\,2^{-j\eps/2}K\,\mu(F_{\beta})\,.
$$
We already said that all terms, in particular, the  analog of the sum $\Sigma_1$ also get $2^{-j\eps/2}$ factor.
This is nice as we get
$$
\sum_{S\in\SSS, F(S)\subset I, S \,\text {is maximal}} a_S^j \leq c\,2^{-j\eps/2} \mu(I)\,.
$$
Now we again need to estimates the whole sum
\begin{equation}
\label{againstarj}
\sum_{S\in\SSS, F(S)\subset I} a_S^j \leq c\,2^{-j\eps/2} \mu(I)\,.
\end{equation}
This achieved exactly as before with the help of \eqref{yellowCarl} of Theorem \ref{yellowS}.
We consider $S_{\al}$ to be maximal $S\in\SSS, F(S)\subset I$, and then for a fixed $\al$ consider 
$S_i(\al)$ to be maximal $S\in\SSS, F(S)\subset S_{\al}$. 

Next generation of stopping cubes will give a contribution $ \frac12 2^{-j\eps/2}\mu(I)$ because 
$$
\sum_{\al}\sum_i\mu(S_i(\al))\leq \frac12\sum_{\al} \mu(S_{\al})\leq \frac12 \mu(I)\,,
$$
yet next generation will come with the contribution  $ \frac14 2^{-j\eps/2}\mu(I)$ et cetera... And we get \eqref{againstarj}.
All this is because of Theorem \ref{yellowS}.

\vspace{.1in}

The proof of Theorem \ref{pivotal3} is finished at last.

\section{The proof of Theorem \ref{Kchithe}. An estimate of $K_{\chi}$ via the weak norm}
\label{Kchith}

We need to estimate the best constant in inequalities
$$
\|T\chi_I w^{-1}\|_w^2 \le B \,w^{-1}(I)\,.
$$
and
$$
\|T'\chi_I w\|_{w^{-1}}^2 \le B \,w(I)\,.
$$
Recall the Lorentz space 
$$
L^{2,1}(w^{-1}):=\{f: \|f\|_{L^{2,1}(w^{-1})}:= \int_0^{\infty} (w^{-1}(x: |f(x)|>t))^{1/2} dt<\infty\}\,.
$$
By definition 
$$
w^{-1}(I)^{1/2} = \|\chi_I\|_{L^{2,1}(w^{-1})}\,.
$$
Therefore the inequalities above can be rewritten in a different fashion:
$$
\|T\chi_I w^{-1}\|_w \le \sqrt{B} \,\|\chi_I\|_{L^{2,1}(w^{-1})}\,.
$$
and

$$
\|T'\chi_I w\|_{w^{-1}} \le \sqrt{B} \,\|\chi_I\|_{L^{2,1}(w)}\,.
$$
The first one can be further rewritten as
$$
\|\frac{T\chi_I w^{-1}}{w^{-1}}\|_{w^{-1}} \le \sqrt{B} \,\|\chi_I\|_{L^{2,1}(w^{-1})}\,.
$$
Therefore, if we denote by $\mathcal{T}$ he operator from $L^{2,1}(w^{-1})$ to $L^2(w^{-1})$ acting by the rule
$$
\mathcal{T}f :=\frac{Tf w^{-1}}{w^{-1}}
$$
we obtain
$$
\sqrt{B}\le \|\mathcal{T}\|\,.
$$
Take a look at the adjoint operator (the duality is with respect to $\int\dots w^{-1}dx$). It is just $f\rightarrow T'f$ from $L^2(w^{-1})$ to $L^{2,\infty}(w^{-1})$.
Therefore
$$
\sqrt{B}\le \max( \|T':L^2(w^{-1})\rightarrow L^{2,\infty}(w^{-1}))\|,\, \|T:L^2(w)\rightarrow L^{2,\infty}(w)\|)
$$

Theorem \ref{Kchithe} is completely proved.

\vspace{.1in}

Combining this with Theorem \ref{pivotal3} we see that Theorem \ref{strongweak} is already proved.

\section{The proof of Theorem \ref{weakpv}. Extrapolation type approach.}
\label{weakpvsec}

This is an abstract theorem, $T$ is arbitrary here such that
\begin{equation}
\label{weak1}
\sup_{t>0}t\,W (\{x: |Tf(x)|>t)\} \le \phi([W]_{A_1})\|f\|_{L^1(W)}
\end{equation} for any $f\in L^1(W)$ and any $W\in A_1$.

We are going to prove that
\begin{equation}
\label{weak2}
\sup_{t>0}t\,w (\{x: |Tf(x)|>t\})^{1/2} \le c\, \phi([w]_{A_2})\|f\|_{L^2(w)}
\end{equation} for any $f\in L^2(w)$ and any $w\in A_2$, where $c$ is depending only on dimension.

Fix such a $w$. We can consider the case $t=1$ only, this is just homogeneity of both parts of 
\eqref{weak2}. 
Denote
$$
\Omega:= \{x: |Tf(x)|>1\}\,.
$$
We can write
\begin{equation}
\label{wOm}
w(\Omega)^{1/2} =\sup_{h\in L^2(w),\,\|h\|_{w}\le 1} |\int h w dx|\,.
\end{equation}
 Consider operator 
 $$
 S_wf:=\frac{M(fw)}{w}\,,
 $$ 
 where $M$ stands for Hardy--Littlewood maximal function. Notice that $S_wf \ge f$.
 
 By Buckley's theorem we know that
 \begin{equation}
 \label{Sw}
 \|S_w: L^2(w)\rightarrow L^2(w)\|\le c\,\natwo\,.
 \end{equation}
 Using this fact we generate
 $$
 Rh:=\sum_{k=0}^{\infty}\frac{S_w^k h}{2^k\|S^k_w\|}\,.
 $$
Then obviously $Rh\ge h$ and
$$
S_w Rh  \le 2\|S_w\| R h\le c\,\natwo\,Rh,.
$$ 

Which, by definition of $A_1$ means
\begin{equation}
\label{wRh}
wRh \in A_1\,,\,\,[wRh]_{A_1} \le  c\,\natwo\,.
\end{equation}

Call $W:=wRh$. Using \eqref{wOm} with appropriate $h$ which almost gives a supremum and using the obvious fact  $Rh\ge h$ we write
$$
w(\Omega)^{1/2} =2 \int h w dx\le \int wRh dx =W(\Omega) = W(\{x: |Tf| >1\})\,.
$$
Combine this with \eqref{weak1}:
$$
w(\Omega)^{1/2}\le W(\{x: |Tf| >1\}) \le \phi([W]_{A_1})\int |f|W dx\,.
$$
Now use the estimate of $A_1$ norm of $W=wRH$, namely, \eqref{wRh}. Then
$$
w(\Omega)^{1/2} \le \phi(c[w]_{A_2})\int |f|wRh dx\,.
$$
Hence,
$$
w(\Omega)^{1/2}\le   \phi(c[w]_{A_2})\|f\|_{w} \|Rh\|_{w}\,.
$$

But
$$
 \|Rh\|_w\le\sum_{k=0}^{\infty}\frac{\|S_w^k h\|_w}{2^k\|S^k_w \|}\le  2\,.
 $$
 Putting this into the previous inequality we obtain $w(\Omega)^{1/2}\le  2 \phi(c[w]_{A_2})\|f\|_{w} \,.$ Theorem \ref{weakpv} is proved.
 
 \section{Discussion}
 \label{discu}
 
 All the results from above can be proved for \cz operators on homogeneous metric spaces.

\end{document}